\documentclass[11pt]{amsart}

 \usepackage{amsfonts,graphics,amsmath,amsthm,amsfonts,amscd,amssymb,amsmath,latexsym,multicol,
 mathrsfs}
\usepackage{epsfig,url}
\usepackage{flafter}
\usepackage{fancyhdr}
\usepackage{hyperref}
\hypersetup{colorlinks=true, linkcolor=black, breaklinks=true}

\usepackage{graphicx}
\usepackage{xcolor}
\usepackage{float}
\usepackage{bm}
\usepackage{enumitem}
\usepackage{multirow}
\usepackage{scalerel,stackengine}
\usepackage{stmaryrd}
\usepackage{datetime}
\usepackage[title]{appendix}
\usepackage{url}
\usepackage{tabulary}
\usepackage{booktabs}
\usepackage{tikz}
\usetikzlibrary{positioning}
\usepackage{verbatim}

\usepackage[utf8]{inputenc}
\usepackage[english]{babel}
\usepackage{csquotes}
\usepackage{comment}

\usepackage[
backend=biber,
style=alphabetic,
sorting=anyt,
doi=false,isbn=false,url=false,
maxalphanames=99,maxbibnames=99,backref=true,
]{biblatex}
\DefineBibliographyStrings{english}{
  backrefpage  = {\hspace{-0.35em}},
  backrefpages = {\hspace{-0.35em}},
}

\DeclareFieldFormat{bracketswithperiod}{\mkbibbrackets{#1}}

\renewbibmacro*{pageref}{%
  \iflistundef{pageref}
    {}
    {\printtext[bracketswithperiod]{%
       \ifnumgreater{\value{pageref}}{1}
         {\bibstring{backrefpages}\ppspace}
         {\bibstring{backrefpage}\ppspace}%
       \printlist[pageref][-\value{listtotal}]{pageref}}%
     }}

 \theoremstyle{plain}
 \newtheorem{theorem}{Theorem}[section]
 \newtheorem{lemma}[theorem]{Lemma}
 \newtheorem{corollary}[theorem]{Corollary}
 \newtheorem{proposition}[theorem]{Proposition}
 
  \newtheorem{question}[theorem]{Question}

 \theoremstyle{definition}
 \newtheorem{definition}[theorem]{Definition}
 \newtheorem{definitionnotation}[theorem]{Notation}
 
 \newtheorem{lemmadef}[theorem]{Definition-Lemma}
 \newtheorem{remark}[theorem]{Remark}

 \theoremstyle{remark}

 \theoremstyle{plain} 
\newcommand{\thistheoremname}{}
\newtheorem{genericthm}[theorem]{\thistheoremname}

  \newtheorem*{genericthm*}{\thistheoremname}
\newenvironment{namedthm*}[1]
  {\renewcommand{\thistheoremname}{#1}%
   \begin{genericthm*}}
  {\end{genericthm*}}
 
\newcommand{\abs}[1]{\left\vert#1\right\vert}

\newcommand{\divisor}{\operatorname{Div}}
\newcommand{\hg}{\operatorname{H}}

\newcommand{\supp}{\operatorname{Supp}}

\newcommand{\shomo}{\operatorname{\mc{H}om}}

\newcommand{\picard}{\operatorname{Pic}}

\newcommand{\tr}{\operatorname{tr}}

\newcommand{\QCQ}{\operatorname{\mathbb{Q}-Car}}
\newcommand{\CQ}{\operatorname{Car}}
\newcommand{\Hpicard}{\operatorname{HPic}}

\newcommand{\mni}{\medskip\noindent}

\newcommand{\mbb}{\mathbb}
\newcommand{\QQ}{\mbb{Q}}
\newcommand{\NN}{\mbb{N}}
\newcommand{\ZZ}{\mbb{Z}}
\newcommand{\CC}{\mbb{C}}
\newcommand{\RR}{\mbb{R}}

\newcommand{\mc}{\mathcal}

\newcommand{\wt}{\widetilde}

\newcommand{\ol}{\overline}

\usepackage{graphicx}
\usepackage{pifont}
\usepackage{xcolor}
\usepackage{float}
\usepackage{amssymb}
\usepackage{amsmath}
\usepackage{mathrsfs}
\usepackage{bm}
\usepackage[all,cmtip]{xy}
\usepackage{enumitem}
\usepackage[utf8]{inputenc}
\usepackage[english]{babel}
\usepackage{multirow}
\usepackage{mathtools}
\usepackage{scalerel,stackengine}
\usepackage{stmaryrd}
\usepackage{datetime}
\usepackage[title]{appendix}
\usepackage{url}
\usepackage{tabulary}
\usepackage{booktabs}
\usepackage{tikz}
\usetikzlibrary{positioning}
\usepackage{verbatim}

\DeclarePairedDelimiter\ceil{\lceil}{\rceil}
\DeclarePairedDelimiter\floor{\lfloor}{\rfloor}

\title{\large I\MakeLowercase {njectivity} T\MakeLowercase{heorem for} G\MakeLowercase{eneralised} P\MakeLowercase{airs} 
\MakeLowercase{on} S\MakeLowercase{urfaces}}
\author{\large S\MakeLowercase{antai} \large Q\MakeLowercase{u}}
\begin{document}

\begin{abstract}
	In this paper we study the question asked by Caucher Birkar 
    about injectivity theorem on cohomologies of generalised pairs.
	By applying techniques from complex analytic geometry,
	we show that the injectivity theorem holds
	for generalised Kawamata log terminal (g-klt) generalised pairs on projective surfaces.
\end{abstract}

\maketitle

\setcounter{tocdepth}{1}
\tableofcontents

\section{Introduction}

\mni
{\textbf{\sffamily{Injectivity theorem for generalised pairs.}}}
One of the most fundamental results for complex projective varieties is Koll\'ar's injectivity theorem 
\cite[Theorem 2.2]{Kol-direct-image}, from which
Kodaira's vanishing theorem can be deduced very quickly
by Serre's vanishing theorem.
\begin{theorem}[\protect{\cite[Theorem 2.2]{Kol-direct-image}}]\label{Kol-original}
	Let $X$ be a smooth projective complex variety and $\mc{L}$ a line bundle on $X$.
	Let $s$ be a section of $\mc{L}^k$ for some $k\in \NN$, and $D$ the divisor of $s$.
	The multiplication by $D$ defines a homomorphism on cohomologies, denoted by
	\[ \phi_D^i\colon \hg^i(X, \omega_X\otimes \mc{L}^n) \to \hg^i(X, \omega_X\otimes \mc{L}^{n+k}). \]
	Assume that some power of $\mc{L}$ is generated by global sections, i.e., 
	$\mc{L}$ is semi-ample.  Then $\phi_D^i$ is injective for every $i\ge 0$
	and every $n\ge 1$.
\end{theorem}


Beyond the category of smooth projective varieties, Koll\'ar's injectivity theorem
also holds for projective Kawamata log terminal (klt) pairs, and it is applied to prove
the log abundance theorem for threefolds (see \cite[\S 7]{KMM94}).

\begin{theorem}[\protect{\cite[Lemma 7.3]{KMM94}}]\label{K-inj}
	Suppose that $(X, \Delta)$ is a projective klt pair, and $L$ is an integral $\mathbb{Q}$-Cartier
	divisor such that $L-(K_X +\Delta)$ is semi-ample.  Suppose that $D$ and $D'$ 
	are effective integral $\mathbb{Q}$-Cartier divisors such that $D+D'\in \abs{ m(L- (K_X+\Delta))}$
	for some $m\in \mathbb{N}$.  Denote by
	\[ \varphi_D\colon \mc{O}_X(L) \to \mc{O}_X(L + D)  \]
	the multiplication morphism induced by $D$ (see Notation~\ref{multi-D}).
	Then the homomorphisms on cohomologies induced by $\varphi_D$,
	\[\hg^i(\varphi_D)\colon \hg^i(X, \mc{O}_X(L)) \to \hg^i(X, \mc{O}_X(L+D)),\]
	are injective for all $i\ge 0$.
\end{theorem}

The injectivity theorems
(Theorem~\ref{Kol-original} and Theorem~\ref{K-inj}) have many applications in algebraic geometry and 
minimal model program (MMP) (see \cite{KMM94}, \cite{Fuj17-inj-thms}, \cite{Fuj-survey}).
Moreover, generalisations of Theorem~\ref{Kol-original} and Theorem~\ref{K-inj} also apply to other categories 
of geometric objects.  For example, \cite{Fuj17-inj-thms} studies the injectivity theorem
and its applications for simple normal crossing varieties (resp. pairs);
\cite{Fujino-M21} studies variants of the injectivity theorem in the complex analytic setting.

On the other hand, \emph{generalised pairs}
(see Definition~\ref{g-pairs-defn}), a geometric structure
generalising usual pairs (see \cite{km98}) introduced in \cite{BirkarZhang}, play a very important role in recent developments 
of birational geometry, see \cite{birkar-g-pairs-in-BG}.
Roughly speaking, a generalised pair
$(X, B + \mathbf{M})$ consists of a normal variety $X$,
an effective divisor $B$, and a b-divisor $\mathbf{M}$
that descends to a nef divisor on a birational model of $X$.
Singularities of generalised pairs are defined in
a similar fashion for usual pairs, see \cite[page 361]{B-Fano}, 
so we can talk about \emph{generalised lc (g-lc)}, \emph{generalised klt (g-klt)},
and \emph{generalised dlt (g-dlt)} generalised pairs.

Generalised pairs appear naturally in birational geometry, for example,
the canonical bundle formula (see \cite[(3.4)]{B-Fano}).  
Caucher Birkar asked the question whether the following generalised version 
of the injectivity theorem
for g-klt $\QQ$-g-pairs holds.

\begin{question}[Birkar]\label{qu-bir}
	Suppose that $(X, B+\mathbf{M})$ is a projective g-klt $\QQ$-g-pair.  
	Assume that $L$ is an integral $\mathbb{Q}$-Cartier
	divisor such that
	\[A:= L- (K_X+B+\mathbf{M}_X)\]
	is semi-ample. 	
	Assume that $D$ is an integral effective $\mathbb{Q}$-Cartier divisor 
	such that 
    \[ \hg^0(X, \mc{O}_X(mA-D)) \not = 0 \text{ for some } m\in \NN. \]
	Then the homomorphisms induced by multiplication by $D$,
	\[\hg^i(X, \mc{O}_X(L)) \to \hg^i(X, \mc{O}_X(L+D)),\]
	are injective for all $i\ge 0$.
\end{question}

Notice that Question~\ref{qu-bir} can be reduced quickly 
to usual klt pairs (cf. Lemma~\ref{abundant-to-nef-big})
if the nef part $\mathbf{M}$ is sufficiently positive, e.g.,
$\mathbf{M}$ is b-nef-and-abundant (see \S \ref{b-abnd-divs}).

\begin{proposition}[=Corollary~\ref{abundant-inj-g-pair}]\label{abundant-inj-g-pair-intro}
    Suppose that $(X, B+\mathbf{M})$ is a projective g-klt $\QQ$-g-pair. 
    Assume that $\mathbf{M}$ is b-nef-and-abundant.
    Then the injectivity theorem holds for $(X, B + \mathbf{M})$.
\end{proposition}

By employing techniques from complex analytic geometry (see \S \ref{complex-a-geo}), the main result of this paper is that
the injectivity theorem holds for g-klt $\QQ$-g-pairs
of dimension 2 without any extra assumptions on
the nef parts of the g-pairs.

\begin{theorem}\label{inj-surface-b-iitaka}
    Let $X$ be a normal projective variety of dimension 2, and
    let $(X, B+ \mathbf{M})$ be a projective g-klt $\QQ$-g-pair. 
    Then the injectivity theorem holds for $(X, B+\mathbf{M})$.
\end{theorem}

Compared to Proposition~\ref{abundant-inj-g-pair-intro}, 
the main difficulty to prove Theorem~\ref{inj-surface-b-iitaka}
is to tackle the case when the linear system associated 
to $\mathbf{M}_Y$ is empty 
for any higher birational model $Y$ of $X$.


\mni
{\textbf{\sffamily{Construction of the article.}}}
In Section~\ref{preliminaries-sec} 
we introduce some basic notions that are used in the proof of 
Theorem~\ref{inj-surface-b-iitaka}, such as 
\emph{numerical dimension},
\emph{b-nef-and-abundant divisors}, etc.
Moreover, as the data of generalised pairs
is quite involved, \S \ref{terminologoes-sec}
is devoted to clarifying the terminologies during
the proof of Theorem~\ref{inj-surface-b-iitaka}.
Most of the results in Section~\ref{preliminaries-sec}
are well-known, but we include some of the proofs
for the lack of references in the setting of generalised pairs.

Section~\ref{complex-a-geo} introduces the necessary 
background of complex analytic geometry techniques to 
prove Theorem~\ref{inj-surface-b-iitaka}.
Proposition~\ref{pushf-positivity} is the main result of this 
section, which is a variant of \cite[Lemma 5.4]{CP17}.
We include the necessary details of some complex analytic facts
for the convenience of readers.
Clearly, the results in this section are
well-known to experts in complex analytic geometry.

In Section~\ref{b-abnd-results-sec} 
we show that the injectivity theorem holds if the nef part of 
the g-pair satisfies certain positive conditions.
Some corollaries of Theorem~\ref{b-kodaira-non-negative} 
can be reduced easily to usual klt pairs, but the more 
general result Theorem~\ref{b-kodaira-non-negative}
and its proof
will be used in the proof of Theorem~\ref{inj-surface-b-iitaka}.

Combining all the results explained in
Sections \ref{preliminaries-sec}, \ref{complex-a-geo},
and \ref{b-abnd-results-sec}, we conclude the proof of 
Theorem~\ref{inj-surface-b-iitaka} in Section~\ref{inj-general-surfaces}.


\mni
{\textbf{\sffamily{Acknowledgements.}}}
The main part of this work was finished when I was a postdoc at Yau Mathematical Sciences Centre
at Tsinghua University.
I thank my postdoc mentor Caucher Birkar for introducing this problem and 
for his helpful comments on the proofs.
I would like to thank Christian Schnell for pointing out the results in \cite{CP17}
and for his very helpful discussion about positivity of direct images
during the first International Congress of Basic Science (ICBS 2023) at Beijing.
I am very grateful to Gao Chen, Long Li, and Miao Song for helpful conversations about complex analytic geometry.
I would also like to thank Xiaowei Jiang for reading a draft version of this paper 
and providing many helpful suggestions.
I extend my sincere thanks to the anonymous referees for carefully
reading of this manuscript and for giving many constructive comments 
which substantially helped improving the quality of this paper.
This work was partially supported by the Project of Stable Support 
for Youth Team in Basic Research Field, funded under Chinese Academy
of Sciences grant YSBR-001.


\section{Preliminaries}\label{preliminaries-sec}

We work over an algebraically closed field $\mbb K$ of characteristic zero.
Without loss of generality, we can assume that $\mbb K$ is the field of 
complex numbers $\CC$ (see \cite{FR86-Lefschetz-principle}).
By saying a \emph{scheme}, we mean a separated scheme of finite type over 
$\mathbb{K}$.
A \emph{variety} is an irreducible, reduced, and quasi-projective scheme.

\begin{definition}
	A \emph{contraction} is a projective morphism of schemes $f: X\to Y$
    such that $f_*\mc{O}_X = \mc{O}_Y$; $f$ is not necessarily birational.
    In particular, $f$ has connected fibres.  Moreover, if $X$ is normal, then $Y$ is also normal.
\end{definition}

For the notation and conventions about $\QQ$-divisors and $\RR$-divisors, such as $\QQ$-linearly equivalence, 
ample $\QQ$-divisors, big $\QQ$-divisors, and pseudo-effective $\QQ$-divisors, etc.,
we refer the readers to \cite[\S 3.1]{bchm} and \cite[(2.3)]{B-Fano}.
In particular, for an $\RR$-divisor $B = \sum a_{\Gamma} \Gamma$ on a scheme $X$ where every $\Gamma$
is a prime divisor on $X$ and $a_{\Gamma} \in \RR$, the \emph{support}
of $B$, denoted by $\supp B$, is the closed subset $\bigcup_{a_{\Gamma}\not=0} \Gamma$;
on the other hand, by writing $\floor{\supp B}$, we mean the reduced Weil divisor 
$\sum_{a_{\Gamma}\not=0} \Gamma$ on $X$.
A $\QQ$-divisor is \emph{$\QQ$-effective} if it is $\QQ$-linearly equivalent to
an effective $\QQ$-divisor.

\begin{definition}
    Let $D$ be a $\QQ$-divisor on a normal variety $X$.  We say that $D$ is \emph{semi-ample}
    if the linear system $\abs{mD}$ is base point free for some $m\in \NN$.
    It is clear that $D$ is semi-ample if and only if
    there is a contraction $f\colon X\to Y$ of normal varieties 
    and an ample $\QQ$-divisor $A$ on $Y$ such that $D = f^* A$.
\end{definition}

For singularities of pairs, such as lc, klt, and dlt, we refer 
the readers to \cite[Chapter 2]{km98}.
For the language of \emph{b-divisors}; we refer the readers to \cite{flips_3folds_4folds} and \cite[(2.7)]{B-Fano} 
for more details about b-divisors.

\begin{definition}[\protect{\cite[Definition 1.4]{BirkarZhang}}, \protect{\cite[(2.13)]{B-Fano}}]\label{g-pairs-defn}
    An \emph{$\RR$-generalised pair}, resp. a \emph{$\QQ$-generalised pair}, 
    (\emph{$\RR$-g-pair}, resp. \emph{$\QQ$-g-pair}, for short) 
    $(X/Z, B + \mathbf{M})$ consists of
    \begin{itemize}
    	\item a normal quasi-projective variety $X$ equipped with a projective morphism $X\to Z$,
    	\item an $\RR$-divisor (resp. a $\QQ$-divisor) $B\ge 0$ on $X$, and
    	\item an $\RR$-b-divisor (resp. a $\QQ$-b-divisor) $\mathbf{M}$ over $X$ that descends to a nef$/Z$ $\RR$-Cartier $\RR$-divisor (resp. a nef$/Z$ 
    	    $\QQ$-Cartier $\QQ$-divisor) 
    	    $\mathbf{M}_{X'}$ on some birational model $X'\to X$ of $X$, and
    	\item $K_X + B + \mathbf{M}_X$ is $\RR$-Cartier (resp. $\QQ$-Cartier). 
    \end{itemize}
    We call $B$ the \emph{boundary part} and $\mathbf{M}$ 
    the \emph{nef part} of the $\RR$-g-pair (resp. $\QQ$-g-pair) $(X/Z, B + \mathbf{M})$.
\end{definition}

For singularities of generalised pairs, see \cite[\S 4]{BirkarZhang}.


\subsection{B-nef-and-abundant divisors}\label{b-abnd-divs}

\begin{definition}\label{abundant:divisor}
	Let $D$ be a nef $\QQ$-divisor on a projective normal variety of dimension $n$.
	We say that $D$ is \emph{abundant} (or \emph{good}) if $\kappa (D) = \nu (D)$,
	that is, if its Iitaka dimension $\kappa (D)$ 
    is equal to its \emph{numerical dimension} $\nu (D)$
	which is the largest integer $k \ge 0$ such that $D^k\cdot A^{n-k}\not =0$
	for an ample divisor $A$ (see \cite[\S V.2.a]{nakayama}).
	It is well-known that $0\le \nu (D) \le \dim X$ and that $\kappa (D)\le \nu (D)$, 
	see \cite[Lemma V.2.1]{nakayama} and \cite[page 1067]{num-dim}.
\end{definition}

\begin{remark}
    Although an abundant divisor is nef in Definition~\ref{abundant:divisor}, we usually say that 
    the divisor is \emph{nef and abundant} instead of just saying abundant.	
\end{remark}

\begin{remark}
    For a nef $\QQ$-divisor $D$ on a non-singular projective variety $X$, 
    the numerical dimension $\nu (D)$ is equal to various other numerical 
    Kodaira dimensions:
    \[ \kappa_{\nu}(D) = \kappa_{\sigma}^-(D) = \kappa_{\sigma}^+ (D) = \kappa_{\sigma}(D) = \nu (D). \] 
    We refer the readers to \cite[\S V.2]{nakayama} for details about these numerical Kodaira dimensions.
    If $X$ is a smooth curve, we have $\nu (D) = 0$ if $\deg D = 0$ and $\nu (D) = 1$ if $\deg D >0$
    by Definition~\ref{abundant:divisor}.
    In this case, it is also straightforward to show that $\kappa_{\sigma}(D) = 0$ 
    if and only if $\deg D = 0$
    by \cite[Proposition V.2.7 (2) and (3)]{nakayama}.
\end{remark}


\begin{definition}[\protect{\cite[Notation V.2.24]{nakayama}}]\label{rel-abdt}
	Let $f\colon X\to Y$ be a contraction of normal varieties and $D$ an $\RR$-Cartier divisor on $X$.
	We say that $D$ is \emph{$f$-abundant} (or \emph{abundant over $Y$}, or \emph{abundant$/Y$}) 
	if $D|_F$ is nef and abundant where $F$ is a 
    \emph{very general} fibre of $f$.
\end{definition}


\begin{definition}\label{b-abundant}
    Let $X$ be a normal projective variety and 
    $\bf{M}$ a b-$\RR$-Cartier b-divisor over $X$.
    We say that $\bf{M}$ is \emph{b-nef-and-abundant} if $\bf{M}$ descends to some birational model $Y$ of $X$
    and ${\bf{M}}_Y$ is nef and abundant.
\end{definition}


\begin{definition}\label{fibre-b-divisor}
	Let $f\colon X\to Y$ be a contraction of normal varieties and $\bf M$ a b-$\RR$-Cartier b-divisor on $X$.  
	Let $F$ be a general fibre of $f$.  We define a b-$\RR$-Cartier b-divisor $\mathbf{M}|_F$ as follows.
	Let $\mu\colon X'\to X$ be a nonsingular birational model of $X$ such that $\mathbf M$ descends to an $\RR$-Cartier
	divisor $\mathbf{M}_{X'}$ on $X'$.  Denote by $F'$ a general fibre of $f\circ \mu$.
	Then $F'$ is nonsingular (by generic smoothness), $F$ is normal, and the induced morphism $\mu'\colon F'\to F$
	is birational.  We define $\mathbf{M}|_F$ as the $\RR$-Cartier closure of the $\RR$-Cartier divisor $\mathbf{M}_{X'}|_{F'}$;
	in particular, $\mathbf{M}|_F$ is a b-$\RR$-Cartier b-divisor that descends to $F'$ and 
	$\mu'_* (\tr_{F'} \mathbf{M}|_F) = \mu'_* (\mathbf{M}_{X'}|_{F'}) = (\tr_X\mathbf{M})|_F$.
\end{definition}


\begin{definition}\label{b-rel-abdt}
	Let $f\colon X\to Y$ be a contraction of normal varieties and $\mathbf{M}$ a b-$\RR$-Cartier b-divisor on $X$.  
	Let $F$ be a very general fibre of $f$.  We say that $\mathbf{M}$ is \emph{b-nef-and-abundant over $Y$} 
	(or \emph{b-nef-and-abundant relative to $f$}) if $\mathbf{M}|_F$ is b-nef-and-abundant.
\end{definition}


\begin{definition}\label{b-iitaka}
    Let $X$ be a normal projective variety and $\mathbf{M}$ a b-$\RR$-Cartier b-divisor over $X$.
    We say that $\mathbf{M}$ has \emph{b-Iitaka dimension $k$} if $\mathbf{M}$ descends to
    some birational model $Y$ of $X$ such that $\kappa (\mathbf{M}_Y) = k$.	
\end{definition}


\begin{lemma}[\protect{\cite[Lemma V.2.3 (1)]{nakayama}}]\label{abundant-to-nef-big}
	Let $D$ be a nef and abundant $\RR$-divisor on a non-singular projective variety $X$.
	Then there exist a birational morphism $\mu\colon Y\to X$, a surjective morphism 
	$f\colon Y\to Z$ of non-singular projective varieties,
	and a nef and big $\RR$-divisor $B$ of $Z$ such that $\mu^*D \sim_{\QQ} f^* B$.
\end{lemma}


We collect some useful properties of nef and abundant divisors.

\begin{lemma}[\protect{\cite[Lemma 2.3]{Ein:Popa08}}]\label{pre:nef:abundant}
	Let $X$ be a normal projective variety and $B$ a nef and abundant $\QQ$-divisor on $X$.
	\begin{itemize}
		\item [\emph{(i)}] If $L$ is a globally generated line bundle on $X$ and if $D\in \abs{L}$ is a general divisor, then $B|_D$ is also nef and abundant.
		\item [\emph{(ii)}] If $C$ is another nef and abundant $\QQ$-divisor on $X$, then $B+C$ is also nef and abundant.
		\item [\emph{(iii)}] If $f: Y\to X$ is a surjective morphism from another normal projective variety, then $f^*B$ is nef and abundant.
	\end{itemize}
\end{lemma}


\subsection{Terminologies about the injectivity theorem}\label{terminologoes-sec}

\begin{definitionnotation}\label{multi-D}
	Let $X$ be a normal variety.  Let $L$ and $D$ be Weil divisors on $X$,
	and assume that $D$ is effective.
	We define the \emph{multiplication morphism} of reflexive sheaves
	$\varphi_D\colon \mc{O}_X(L) \to \mc{O}_X(L + D)$ as follows.
	Recall that $\mc{O}_X(L + D)$ is equal to the reflexive sheaf 
    $(\mc{O}_X(L)\otimes \mc{O}_X(D))^{\vee\vee}$, i.e., the double dual of
    the tensor product $\mc{O}_X(L)\otimes \mc{O}_X(D)$.
    Tensoring up the canonical injection 
    $$i_D\colon \mc{O}_X \to \mc{O}_X(D)$$ by $\mc{O}_X(L)$,
    and then composing with the natural morphism of taking double dual,
    $$\alpha\colon \mc{O}_X(L)\otimes \mc{O}_X(D) \to (\mc{O}_X(L)\otimes \mc{O}_X(D))^{\vee \vee},$$
    we get the morphism $\varphi_D$, called the \emph{multiplication morphism induced by $D$}
    (or \emph{multiplication by $D$}, for simplicity):
    \[\xymatrix{
    \mc{O}_X(L)\ar[r]^{i_D\,\,\,\,\,\,\,\,\,\,\,\,\,\,\,\,}\ar@/_1.5pc/[rr]_{\varphi_D} & \mc{O}_X(D)\otimes \mc{O}_X(L)\ar[r]^{\,\,\,\,\,\,\,\,\,\,\alpha} & \mc{O}_X(L+ D).
    }\]
    Moreover, we denote by $\hg^i(\varphi_D)$ the induced morphism on the $i$-th cohomology:
    \[ \hg^i(\varphi_D)\colon \hg^i(X, \mc{O}_X(L)) \to \hg^i(X, \mc{O}_X(L+D)) \text{ for } i\ge 0. \]
\end{definitionnotation}


The following theorem is a generalisation of 
the original form of Koll\'ar's injectivity theorem.

\begin{theorem}[\protect{\cite[Corollary 3.3 (i)]{Ein:Popa08}}, \protect{\cite[Corollary 5.12 (b)]{EsV}}]\label{Ein:Popa}
	Let $L$ be an integral divisor on a smooth projective variety $X$,
	and let $\Delta = \sum_i \delta_i\Delta_i$ be a simple normal crossing $\QQ$-divisor
	with $0<\delta_i <1$ for all $i$.  Assume that $L-\Delta$ is nef and abundant 
	and that $B$ is an effective integral divisor such that $L-\Delta -\varepsilon B$
	is $\QQ$-effective for some rational number $0<\varepsilon <1$.  Then 
	the natural homomorphisms
	\[ \hg^i(X, \mc{O}_X( K_X + L))\to \hg^i(X, \mc{O}_X(K_X + L +B)) \]
	are injective for all $i\ge 0$. 
\end{theorem}


To state our results about injectivity theorems of generalised pairs 
in a more convenient way, we will use the following terminologies.

\begin{definition}
Let $(X, B + \mathbf{M})$ be a $\QQ$-g-pair,
and let $L$ be an integral $\mathbb{Q}$-Cartier divisor such that $A = L- (K_X+B+ \mathbf{M}_X)$
is semi-ample.  Let $D$ be an integral effective $\mathbb{Q}$-Cartier divisor 
such that $\hg^0(X, \mc{O}_X(mA-D))$ is nonzero for some $m\in \NN$
(in particular, $mA = m(L- (K_X + B + \mathbf{M}_X))$ is a Cartier divisor).
We say that \emph{the injectivity theorem holds for $\{(X, B + \mathbf{M}), L, D \}$}
if the homomorphisms of cohomologies induced by multiplication by $D$,
\[\hg^i(\varphi_D)\colon \hg^i(X, \mc{O}_X(L)) \to \hg^i(X, \mc{O}_X(L+D)),\]
are injective for all $i\ge 0$.
\end{definition}


\begin{definition}
	Let $(X, B + \mathbf{M})$ be a $\QQ$-g-pair.  We say that:
	\begin{itemize}
		\item \emph{The injectivity theorem holds for $\{(X, B + \mathbf{M}), L\}$}, where $L$ is 
		      an integral $\mathbb{Q}$-Cartier divisor such that $A = L- (K_X+B+ \mathbf{M}_X)$
              is semi-ample, if for any integral effective $\mathbb{Q}$-Cartier divisor $D$
              such that $\hg^0(X, \mc{O}_X(mA-D))$ is nonzero for some $m\in \NN$, the injectivity theorem
              holds for $\{ (X, B + \mathbf{M}), L, D \}$.
        \item \emph{The injectivity theorem holds for $(X, B + \mathbf{M})$} if the injectivity 
              theorem holds for any $\{(X, B + \mathbf{M}), L\}$ where $L$ is an integral
              $\QQ$-Cartier divisor such that $A = L- (K_X+B+ \mathbf{M}_X)$ is semi-ample.
	\end{itemize}
\end{definition}


The following result shows that to prove the injectivity theorem for the data
$\{(X, B + \mathbf{M}), L, D \}$ we can assume that $D$ is an integral effective Cartier divisor
in the linear system $\abs{m(L - (K_X + B + \mathbf{M}_X))}$ for sufficiently large $m\gg 0$.
The proof follows easily from the arguments in
\cite[Theorem 3.2, page 573]{Kawamata:pluri:system},
so we left it to readers.

\begin{lemma}\label{suff-large-m}
	Let $(X, B + \mathbf{M})$ be a $\QQ$-g-pair,
    and let $L$ be an integral $\mathbb{Q}$-Cartier divisor such that $A = L- (K_X+B+ \mathbf{M}_X)$
    is semi-ample.  Let $D$ be an integral effective $\mathbb{Q}$-Cartier divisor
    such that $\hg^0(X, \mc{O}_X(mA-D))$ is nonzero for some $m\in \NN$.
    If there exists an $n\in \NN$ such that the injectivity theorem holds for 
    $\{(X, B + \mathbf{M}), L, G \}$ where $G \in \abs{nm A}$ is an 
    integral effective Cartier divisor,
    then the injectivity theorem holds for $\{(X, B + \mathbf{M}), L, D \}$.
\end{lemma}


Now, we show that to prove the injectivity theorem for the data
$\{(X, B + \mathbf{M}), L, D \}$ we can assume further that $L$ is an integral Cartier divisor
(not only $\QQ$-Cartier).

\begin{proposition}[cf. \protect{\cite[Proposition 10.18]{Kol-Shafarevich}}]\label{covering}
	Let $(X, B + \mathbf{M})$ be a g-klt $\QQ$-g-pair.
	Let $L$ be an integral $\QQ$-Cartier divisor on $X$, and let $A$ be a $\QQ$-Cartier
	$\QQ$-divisor on $X$ such that
	\[ L = K_X + B + \mathbf{M}_X + A. \]
	Let $p\colon Y\to X$ be a log resolution of $(X, \floor{\supp B} + \floor{\supp L})$ 
	on which $\mathbf{M}$ descends to a nef $\QQ$-divisor 
	$\mathbf{M}_Y$.  Then there exist a Cartier divisor $L_Y$ on $Y$ 
	and a simple normal crossing divisor
	$B_Y$ on $Y$ such that
	\begin{itemize}
		\item [\emph{(1)}] $(Y, B_Y + \ol{\mathbf{M}_Y})$ is a g-klt $\QQ$-g-pair,
		\item [\emph{(2)}] $L_Y = K_Y + B_Y + \mathbf{M}_Y + p^* A$, 
		\item [\emph{(3)}] $p_*\mc{O}_Y(L_Y) = \mc{O}_X(L)$, and
		\item [\emph{(4)}] $R^ip_* \mc{O}_Y(L_Y) = 0$ for every $i>0$.
	\end{itemize}
\end{proposition}

\begin{proof}
    Let $\{E_i\}$ be the set of exceptional divisors of $p$.
	By assumptions, we can write
	\[ K_Y + p_*^{-1}B + \sum a_i E_i + \mathbf{M}_Y = 
	p^*(K_X + B + \mathbf{M}_X)  \]
	where $a_i < 1$ since $(X, B + \mathbf{M})$ is g-klt.  Define the rational numbers $d_i\ge 0$ by
	\[ p^* L = \floor{p^* L} + \sum d_i E_i. \]
	Set
	\begin{align*}
		L_Y &:= \floor{p^* L} + \sum \ceil{-a_i + d_i} E_i, \text{ and }\\
		B_Y &:= p_*^{-1}B + \sum \langle a_i - d_i \rangle E_i.
	\end{align*}
	It is easy to see that (1) and (2) are satisfied by construction.
	Condition (3) follows from \cite[Lemma 7.11]{Debarre_higher_dim_ag}  
	and \cite[Lemma II.2.11]{nakayama}.  Moreover, as $\mathbf{M}_Y$ is nef,
	\[ L_Y - (K_Y + B_Y) = \mathbf{M}_Y + p^* A \]
	is nef and big over $X$.  Then (4) follows from the relative 
	Kawamata-Viehweg vanishing theorem (see \cite[Theorem 1-2-5, Remark 1-2-6]{kmm}).
\end{proof}


To state the arguments in our proofs more conveniently, we summarise 
the properties of the log resolution in Proposition~\ref{covering}
in the following definition.

\begin{definition}\label{admissible-blowing-up}
	Let $(X, B + \mathbf{M})$ be a g-klt $\QQ$-g-pair.
	Let $L$ be an integral $\QQ$-Cartier divisor on $X$, and let $A$ be a $\QQ$-Cartier
	$\QQ$-divisor on $X$ such that
	$ L = K_X + B + \mathbf{M}_X + A$.
	Let $(Y, B_Y + \mathbf{N})$ be a g-klt $\QQ$-g-pair and $L_Y$ an integral 
	$\QQ$-Cartier divisor on $Y$.
	We say that $\{ (Y, B_Y + \mathbf{N}), L_Y \}$ is an \emph{admissible blowing up} 
	of $\{ (X, B + \mathbf{M}), L \}$ if there is a 
	birational contraction $p\colon Y\to X$ such that
	\begin{itemize}
		\item $L_Y = K_Y + B_Y + \mathbf{N}_Y + p^* A$, 
		\item $\mathbf{M}$ is isomorphic to $\mathbf{N}$ as b-divisors on $Y$,
		\item $p_*\mc{O}_Y(L_Y) = \mc{O}_X(L)$, and
		\item $R^ip_* \mc{O}_Y(L_Y) = 0$ for every $i>0$.
	\end{itemize}
\end{definition}


\begin{proposition}\label{reduce-to-cartier}
	Let $\mc{Q}$ be a collection of $\QQ$-g-pairs such that the set
	\[ \QCQ (\mc{Q}) := \bigg\{ \{ (X, B + \mathbf{M}), L\}\,\bigg|\,
	\begin{array}{c}
	    (X, B + \mathbf{M}) \in \mc{Q} 
	    \text{ and } L \text{ is an integral } \QQ\text{-Cartier divisor} \\
	    \text{such that } L - (K_X + B + \mathbf{M}_X) \text{ is semi-ample}
	\end{array}\bigg\}\]
	is closed under admissible blowing ups.  Moreover, denote by
	\[ \CQ (\mc{Q}) := \bigg\{ \{ (X, B + \mathbf{M}), L\}\in \QCQ(\mc{Q}) \,\bigg|\,
     L \text{ is an integral } \text{Cartier divisor} \bigg\}.\]
	Then the injectivity theorem holds for every $\{ (X, B + \mathbf{M}), L\}$ in $\QCQ (\mc{Q})$
	if and only if the injectivity theorem holds for every 
	$\{ (X, B + \mathbf{M}), L\}$ in $\CQ (\mc{Q})$.
\end{proposition}

\begin{proof}
	Assume that the injectivity theorem holds for every member in $\CQ (\mc{Q})$.
	Let $\{ (X, B + \mathbf{M}), L \}$ be an arbitrary member in $\QCQ (\mc{Q})$.
	Let $p\colon Y\to X$ be a log resolution such that $\{ (Y, B_Y + \ol{\mathbf{M}_Y}), L_Y \}$
	is an admissible blowing up as in Proposition~\ref{covering};
	in particular, $\ol{\mathbf{M}_Y}$ descends to a nef $\QQ$-divisor on $Y$, 
	and $L_Y$ is a Cartier divisor on $Y$.
	By Lemma~\ref{suff-large-m}, to show that the injectivity
	theorem holds for $\{(X, B + \mathbf{M}), L\}$, it suffices to take a Cartier divisor
	$G \in \abs{m(L-(K_X + B + \mathbf{M}_X))}$ for sufficiently large $m\gg 0$
	and to show that the injectivity theorem holds for $\{(X, B + \mathbf{M}), L, G\}$.
	
	Denote by $A_Y$ the semi-ample divisor $p^* A$.  Then
	$G_Y := p^* G$ is a Cartier divisor in the linear system $\abs{mA_Y}$.
	Consider the following commutative diagram of natural morphisms,
	\[\xymatrix{
	\mc{O}_X(L)\ar[r]^{\gamma}\ar[d]_{\varphi_G} & p_*\mc{O}_Y(L_Y)\ar[d]^{p_*(\varphi_{G_Y})} \\
	\mc{O}_X(L+G)\ar[r]_{\gamma_G\,\,\,\,\,\,\,\,} & p_*\mc{O}_Y(L_Y + G_Y)
	}\]
	where $\gamma$ is the isomorphism as in Proposition~\ref{covering} (3), 
	$\varphi_G$ and $\varphi_{G_Y}$ are the multiplication morphisms
	as in Notation~\ref{multi-D}, and
	$\gamma_G$ is obtained by tensoring up $\gamma$ with the invertible sheaf 
	$\mc{O}_X(G)$ by projection formula.
	Note that $\gamma_G$ is also an isomorphism. 
	Moreover, we also have
	\[ R^ip_* \mc{O}_Y(L_Y + G_Y) = 0 \text{ for every }i>0 \]
	by the same argument as in the proof of Proposition~\ref{covering}.
	Then we have the following induced commutative diagram of cohomologies
	for $i\ge 0$.
	\[\xymatrix{
	\hg^i(X, \mc{O}_X(L))\ar[r]^{\hg^i(\gamma)}\ar[d]_{\hg^i(\varphi_G)} & \hg^i(Y, \mc{O}_Y(L_Y))\ar[d]^{\hg^i(\varphi_{G_Y})} \\
	\hg^i(X, \mc{O}_X(L+G))\ar[r]_{\hg^i(\gamma_G)\,\,} & \hg^i(Y, \mc{O}_Y(L_Y + G_Y))
	}\]
	By assumption, $\hg^i(\varphi_{G_Y})$ is injective for every $i\ge 0$.
	On the other hand, by Leray spectral sequence, $\hg^i(\gamma)$ and $\hg^i(\gamma_G)$
	are isomorphisms (for Proposition~\ref{covering} (3)).  
	Therefore, $\hg^i(\varphi_G)$ is injective for
	every $i\ge 0$, that is, the injectivity theorem holds for $\{(X, B + \mathbf{M}), L, G\}$.
\end{proof}


\begin{lemma}\label{red-kod-dim-not-zero}
	Suppose that $(X, B+\mathbf{M})$ is a projective g-klt $\QQ$-g-pair. 
	Assume that $L$ is an integral $\mathbb{Q}$-Cartier
	divisor such that
	$A = L- (K_X+B+ \mathbf{M}_X)$
	is semi-ample.  If $\kappa (A) = 0$, then the injectivity theorem holds for $\{(X, B + \mathbf{M}), L\}$.
\end{lemma}

\begin{proof}
    By Lemma~\ref{suff-large-m}, it suffices to prove that the injectivity theorem holds
    for $\{ (X, B + \mathbf{M}), L, D \}$ where $D$ is an effective Cartier divisor
    in the linear system $\abs{mA}$ for sufficiently large $m\gg 0$.
	Let $f\colon X\to Z$ be the contraction induced by the semi-ample divisor $A$, so that 
	$A = f^* A_Z$ for an ample $\QQ$-divisor $A_Z$ on $Z$.
	By assumption, $\dim Z = \kappa (A_Z) = \kappa (A) = 0$ (see \cite[Theorem 2.1.27]{robpositivity1}),
	which implies that $D = 0$.  Thus, the multiplication morphism induced by $D$
	(see Notation~\ref{multi-D}) is an isomorphism, so the injectivity theorem holds.
\end{proof}

Because of Lemma~\ref{red-kod-dim-not-zero}, we assume that $\kappa (A) \ge 1$ in the rest of this paper.


\section{Elementary complex analytic geometry}\label{complex-a-geo}

The results in this section are well-known in complex analytic geometry.
To clarify the notions used in our arguments, 
we give some of the full definitions and sketch the proofs.
In particular, we give the most down-to-earth description of $\QQ$-line bundles
(and integral line bundles) for the convenience of local computations.
All varieties in this section are quasi-projective complex analytic varieties;
in particular, all complex manifolds we consider are K\"ahler.
We remark that the topology on complex analytic varieties is the analytic topology in this section.
However, the notions \emph{general fibres} and \emph{very general fibres}
are still defined in Zariski topology. 

We first recall the construction of (integral) line bundles on smooth complex varieties
(equivalently, complex manifolds).  Let $X$ be a smooth complex variety.
As an analytic variety, $\mc{O}_X$ is the sheaf of holomorphic functions on $X$.
By a \emph{vector bundle} on $X$, we mean a locally free coherent sheaf of finite rank on $X$.
Let $\mc{L}$ be a line bundle on $X$.  A \emph{trivialisation of $\mc{L}$}
consists of the data $\{ (U_i, \phi_i, g_{ij}) \}$, where $\{U_i\}$
is an open covering of $X$ such that $\mc{L}|_{U_i}$ is trivial on every $U_i$,
$\phi_i\colon \mc{L}|_{U_i} \to \mc{O}_{U_i}$ is an isomorphism of $\mc{O}_{U_i}$-modules,
and $\phi_i\circ \phi_j^{-1}$ is the multiplication by the nowhere vanishing holomorphic function 
$g_{ij}\in \mc{O}_X(U_i\cap U_j)$.  We call $\{g_{ij}\}$ the \emph{transition functions}
of the trivialisation.  Given two trivialisations $\{ (U_i, \phi_i, g_{ij})\}$
and $\{ (V_i, \lambda_i, h_{ij}) \}$, we can assume that the coverings $\{U_i\}$ and $\{V_i\}$
are the same up to taking refinements.  Then we say that the two trivialisations are
\emph{equivalent} if there is a collection $\{(U_i, \rho_i)\}$ of nowhere vanishing holomorphic functions
$\rho_i\in \mc{O}_X(U_i)$ such that the isomorphism $\lambda_i\circ \phi_i^{-1}$
on $\mc{O}_{U_i}$ is given by multiplication of $\rho_i$ and that 
$h_{ij} = \rho_i g_{ij} \rho_j^{-1}$ on $U_i\cap U_j$.
The equivalence classes of trivialisations on $X$ are one-to-one corresponding 
to the isomorphism classes of line bundles on $X$.


\subsection{$\QQ$-line bundles and $\QQ$-divisors}

\begin{definition}[\protect{\cite[Definition 2.1]{Kim2010}}]\label{pre-Q-line-bdl}
	Let $X$ be an integral complex variety.  A \emph{$\QQ$-line bundle} $\mc{L}$ on $X$
	is a collection of nowhere vanishing holomorphic functions $g_{ij}\in \mc{O}_X(U_i\cap U_j)$ 
	(called \emph{transition functions})
	on an open covering $\{U_i\}$ of $X$ such that there exists an integer $m\ge 1$
    and $\{(U_i, g_{ij}^m)\}$ defines a line bundle on $X$ (which we denote by $\mc{L}^{\otimes m}$).
    Given two $\QQ$-line bundles $\{(U_i, g_{ij})\}$ and $\{(V_i, h_{ij})\}$,
    up to taking refinements, we can assume that $\{U_i\}$ is equal to $\{V_i\}$.
    The two $\QQ$-line bundles are called \emph{equivalent} if 
    there exists $n\in \NN$ such that $\{ (U_i, g_{ij}^n) \}$ and 
    $\{ (U_i, h_{ij}^n) \}$ define isomorphic line bundles.
\end{definition}


\begin{definition}\label{Q-line-bdl}
	Let $X$ be an integral complex variety and $\mc{L}$
    a $\QQ$-line bundle on $X$.
    The \emph{index of $\mc{L}$} is the positive integer
    \[ m_{\mc{L}} := \min\{ m\in \NN\,|\,\mc{L}^{\otimes m}\text{ is a line bundle} \}. \]
    Similarly, let $D$ be a $\QQ$-divisor on $X$.
    The \emph{index of $D$} is the smallest positive integer $m_D$
    such that $m_D D$ is an integral Weil divisor.
\end{definition}

If $m_{\mc{L}} = 1$, the $\QQ$-line bundle $\mc{L}$ is just a line bundle in the usual sense,
which we will call an \emph{integral line bundle} in the rest of this paper.
For a $\QQ$-line bundle $\mc{L}$,
denote by $\NN (\mc{L})$ the set of all integers $m\in \NN$ such that $\mc{L}^{\otimes m}$
is an integral line bundle.  It is easy to see that 
$m_{\mc{L}}$ is equal to $\gcd \NN (\mc{L})$, the greatest common
divisor of integers in $\NN (\mc{L})$.
If $[\mc{L}]$ is an equivalence class of $\QQ$-line bundles, we denote by
$m_{[\mc{L}]}$ the minimum of indices $m_{\mc{L}'}$ with $\mc{L}'\in [\mc{L}]$.


For an integral complex variety $X$, we denote by $\divisor_{\QQ}(X)$ the group of 
$\QQ$-linearly equivalence classes of $\QQ$-divisors on $X$.
Let $D$ be a $\QQ$-divisor on $X$.
Denote by $\NN (D)$ the set of integers $m\in \NN$ such that $m D$
is an integral Weil divisor.
Then $m_D$ is equal to $\gcd \NN (D)$.
For a $\QQ$-divisor class $[D] \in \divisor_{\QQ} (X)$, we
denote by $m_{[D]}$ the minimum of indices $m_{E}$
such that $E\sim_{\QQ} D$.


\begin{lemma}\label{Q-div-Q-line-bdl}
    Let $X$ be a smooth complex variety.
    Denote by $\mc{P}$ the group of equivalence classes of $\QQ$-line bundles on $X$.
    Then there is an isomorphism of groups
    \[ \Xi\colon \divisor_{\QQ}(X) \to \mc{P} \]
    that preserves indices.
\end{lemma}

\begin{proof}
	As $X$ is a smooth complex variety, we view $X$
	as a complex manifold in the rest of the proof.
	First, let $D$ be a $\QQ$-divisor on $X$. 
	Then $m_D D$ is an integral Weil divisor.
	Since $X$ is smooth, the integral divisor $m_D D$ corresponds to
	an integral line bundle $\mc{F} := \mc{O}_X(m_D D)$.
	By taking refinements of the local coordinate charts of the complex manifold $X$,
	we can take an open covering $\{U_i\}_{i\in I}$ trivialising the line bundle $\mc{F}$
	such that every intersection $U_i\cap U_j$ is diffeomorphic to the Euclidean 
	space $\RR^{2\dim X}$ (see \cite[Theorem I.5.1, Corollary I.5.2]{Bott-Tu}).  
	In particular, the topological fundamental
	group $\pi_1(U_i\cap U_j)$ is trivial for any $i, j\in I$.
	Denote by $\{h_{ij}\}$ a set of transition functions of the isomorphism class $[\mc{F}]$ of 
	the line bundle $\mc{F}$
	with respect to the open covering $\{U_i\}$.
	Then every $h_{ij}$ is a holomorphic function taking values in $\CC\setminus\{0\}$.
	As $U_i\cap U_j$ is simply connected, there is a 
	holomorphic function $g_{ij}\in\mc{O}_X(U_i\cap U_j)$
	such that $g_{ij}^{m_D} = h_{ij}$.  By Definition~\ref{pre-Q-line-bdl},
	$\{(U_i, g_{ij})\}$ defines a $\QQ$-line bundle $\mc{L}$
	such that $m_{\mc{L}}\le m_D$.
	We set $\Xi (D) = [\mc{L}]$.
	
	Let $\{ (V_i, \gamma_{ij})\}$ be another set of transition functions of $[\mc{F}]$,
	so $\{(U_i, g_{ij})\}$ and $\{(V_i, \gamma_{ij})\}$ are equivalent.
	If $\mc{L}'$ is the $\QQ$-line bundle induced by $\{(V_i, \gamma_{ij})\}$,
	it is easy to see that $\mc{L}$ is equivalent to $\mc{L}'$, hence 
	$\Xi$ does not depend on the choice of transition functions of $[\mc{F}]$.
	Moreover, pick another $m\in \NN(D)$, which is a positive multiple of $m_D$.
	Taking transition functions of $\mc{F}^{m/m_D}$ shows that $\Xi$ 
	does not depend on the choice of $m\in \NN(D)$ neither.
	Additionally, let $E\in [D]$, and let $m\in \NN$ such that $m E \sim m D$.
	We have that $\Xi (D) = \Xi (E)$, so $\Xi$ depends only on the $\QQ$-divisor class of $D$.
	Thus, we can write $\Xi ([D]) = [\mc{L}]$.
	Furthermore, by taking $E\in [D]$ satisfying $m_E = m_{[D]}$, we see that
	\begin{equation}
		m_{[\mc{L}]}\le m_{\mc{L}}\le m_{[D]} \le m_D.\label{ineq-1}
	\end{equation}
	
	Conversely, let $\mc{L}'$ be a $\QQ$-line bundle on $X$ with index $m_{\mc{L}'}$,
	i.e., $m_{\mc{L}'} \mc{L}'$ is an integral line bundle.
	Denote by $G$ an integral divisor corresponding to $m_{\mc{L}'} \mc{L}'$.
	We set $[D']\in \divisor_{\QQ}(X)$ as the $\QQ$-divisor class $[G/m_{\mc{L}'}]$.
	Then it is easy to see that $\Xi ([D']) = [\mc{L}']$,
	so $\Xi$ is an isomorphism. 
	Moreover, notice that we can take $\mc{L}'$ such that $m_{\mc{L}'} = m_{[\mc{L}']}$,
	hence 
	\begin{equation}
		m_{[D']} \le m_{D'} \le m_{[\mc{L}']}. \label{ineq-2}
	\end{equation}
	So, $m_{[D]} = m_{[\mc{L}]}$
	for any $\Xi ([D]) = [\mc{L}]$ by \eqref{ineq-1}, \eqref{ineq-2}, 
    and by the fact that $\Xi$ is an isomorphism.
\end{proof}


\begin{definitionnotation}
    Let $X$ be a smooth complex variety.
	Because of Lemma~\ref{Q-div-Q-line-bdl} 
	we can denote by $\picard(X)_{\QQ} := \picard (X)\otimes_{\ZZ} \QQ$ the group of 
	equivalence classes of $\QQ$-line bundles on $X$.
\end{definitionnotation}


\subsection{Singular hermitian metrics on $\QQ$-line bundles}

Now, we collect some necessary notions about singular hermitian metrics.
The main references of this section are \cite{Kim2010} and \cite[Chapters D and E]{Schnell18}.
For technical notions that are not used in our proofs we will just refer 
the readers to the relevant literature.

\begin{definition}[cf. \protect{\cite[Definition 2.3]{Kim2010}}]\label{h-metric}
	Let $X$ be a complex manifold and $\mc{L} = \{(U_i, g_{ij})\}$ a $\QQ$-line bundle on $X$.
	A \emph{singular hermitian metric} $h$ of $\mc{L}$ is a collection of measurable functions 
	$\varphi_i \colon U_i  \to [-\infty, +\infty]$ such that 
	\begin{equation}
		e^{-\varphi_i} = \abs{g_{ij}}^{-2} e^{-\varphi_j} \text{ on }U_i\cap U_j.\label{hermitian-transition}
	\end{equation}
	We call the data
	$$(\mc{L}, h) := \{ (U_i, g_{ij}, \varphi_i) \} $$ 
	a \emph{hermitian $\QQ$-line bundle} on $X$.
	Denote by $\varphi$ the data $\{ (U_i, \varphi_i) \}$.
	We call $\varphi$ the \emph{local weight functions of $(\mc{L}, h)$} (or, \emph{of $h$}, for simplicity).
\end{definition}

\begin{lemmadef}\label{equi-h-metric}
	Let $X$ be a complex manifold.  Two hermitian $\QQ$-line bundles
    $$(\mc{L},h_{\mc{L}}) = \{(U_i, g_{ij}, \varphi_i)\} \text{ and } 
    (\mc{G}, h_{\mc{G}}) = \{(V_i, \lambda_{ij}, \psi_i)\}$$ are called \emph{equivalent}
    if (up to taking refinements so that $\{U_i\}$ is the same as $\{V_i\}$) 
    there are nowhere vanishing holomorphic functions $\rho_i\in \mc{O}_X(U_i)$
    such that 
    \begin{itemize}
        \item [(i)] $\psi_i = \varphi_i + \log \abs{\rho_i}^2$ for every $i$, and
    	\item [(ii)] $\lambda_{ij}^m = \rho_i^m g_{ij}^m \rho_j^{-m}$ on every $U_i\cap U_j$ 
    	      for some $m\in \NN$ such that 
    	      both $\mc{L}^{\otimes m}$ and $\mc{G}^{\otimes m}$ are integral line bundles.
    \end{itemize}
    Denote by $\Hpicard(X)_{\QQ}$ the group of equivalence classes of hermitian
    $\QQ$-line bundles on $X$.  Then by (ii), there is a homomorphism of groups
    \[ \mc{H}\colon \Hpicard(X)_{\QQ}\to \picard (X)_{\QQ} \text{ given by } 
    [(\mc{L}, h_{\mc{L}})] \mapsto [\mc{L}]. \] 
    In particular, for every $[\mc{L}]\in \picard (X)_{\QQ}$, we call an element
    in the fibre $\mc{H}^{-1}([\mc{L}])$, denoted by $h_{[\mc{L}]}$, a \emph{hermitian class
    of the $\QQ$-line bundle class $[\mc{L}]$}.
\end{lemmadef}


If $m_{\mc{L}} = 1$, then Definition~\ref{h-metric} agrees with the usual definition of singular
hermitian metrics on line bundles (see \cite[\S 13]{Schnell18}).
More precisely, suppose that $s$ is a local holomorphic section of $\mc{L}$.
We can assume that $s$ is defined on some open subset $U\subseteq X$
contained in one of the open subsets $U_i$ in Definition~\ref{h-metric}.
Let $f\in \mc{O}_X(U)$ be the image of $s$ under the trivialisation corresponding
to $\{(U_i, g_{ij})\}$.
Then the \emph{squared length of $s$} (with respect to the 
singular hermitian metric $h$) is given by
\[ \abs{s}_h^2 = \abs{f}^2 e^{-\varphi_i}. \]


Unlike integral line bundles, there is no well-defined notion of local holomorphic sections
for $\QQ$-line bundles.  Instead, we can evaluate the singular hermitian metric
of a $\QQ$-line bundle on its multi-valued sections.

\begin{definition}[\protect{\cite[Definition 2.2]{Kim2010}}]\label{m-sect}
	Let $X$ be a complex manifold and $\mc{L}$ a $\QQ$-line bundle on $X$ defined 
	by the data $\{(U_i, g_{ij})\}$ such that $\mc{L}^{\otimes m}$ is an integral line bundle 
	for some $m\in \NN$.  A \emph{multi-valued section} $s$ of $\mc{L}$
	is a collection of sections $f_i\in \mc{O}_X(U_i)$ such that
	\begin{equation}
		g_{ij}^m f_j^m = f_i^m \text{ for every }i,j.\label{h-metric-m-sect}
	\end{equation}
\end{definition}

Now, let $ (\mc{L}, h) = \{(U_i, g_{ij}, \varphi_i)\}$ be a hermitian 
$\QQ$-line bundle on a complex manifold $X$, and
let $s = \{(U_i, f_i)\}$ be a multi-valued section of $\mc{L}$.
Taking norms on \eqref{h-metric-m-sect} shows that 
\begin{equation}
	\abs{g_{ij}}^2 \abs{f_j}^2 = \abs{f_i}^2 \text{ for every }i,j.\label{transition-function}
\end{equation}
On every open subset $U_i$, we define the \emph{squared length of $s$} as
\[ \abs{s}_h^2 := \abs{f_i}^2 e^{-\varphi_i}. \]
By \eqref{hermitian-transition} and \eqref{transition-function}, $\abs{s}_h^2$
is well-defined on the whole $X$.
Similarly, for every open subset $U\subseteq X$ and for a multi-valued section $s_U$
of $\mc{L}|_U$, the squared length $\abs{s_U}_h^2$ is well-defined.


\subsection{Curvature and multiplier ideal sheaf of $\QQ$-line bundles}

For the definition of plurisubharmonic functions on complex manifolds, see \cite[\S 12]{Schnell18} 
and \cite[\S 2.9]{pluri-p-theory}.

\begin{definition}[cf. \protect{\cite[page 1441]{Kim2010}} and \protect{\cite[\S 13]{Schnell18}}]\label{cur-form}
    Let $X$ be a complex manifold, and let $(\mc{L}, h)$ be a hermitian
    $\QQ$-line bundle with local weight functions $\varphi = \{(U_i, \varphi_i)\}$.
    When $\varphi_i\in L_{\text{loc}}^1(U_i)$ for every $i$, 
    the \emph{curvature current} $\Theta_h(\mc{L})$ of $(\mc{L}, h)$
    is defined, in the sense of currents, by the formula
    \[ \Theta_h(\mc{L}) := \frac{\sqrt{-1}}{2\pi}\partial \ol\partial \varphi_i \text{ on each }U_i,  \]
    which is then a globally well-defined $(1,1)$-current on $X$.
    We say that $(\mc{L}, h)$ has \emph{semi-positive curvature} if $\Theta_h (\mc{L}) \ge 0$
    (in the sense of currents),
    which is equivalent to saying that the local weight functions $\varphi$ are plurisubharmonic
    (up to modifying $\varphi$ on a set of measure zero, see \cite[\S 13]{Schnell18}).
\end{definition}

If $m_{\mc{L}} = 1$ (i.e., $\mc{L}$ is an integral line bundle), then $\Theta_h (\mc{L})$
can be expressed in the following way: for every local holomorphic section $f$ of 
$(\mc{L}^{\vee}, h^{\vee})$ we have 
$$ \Theta_h (\mc{L}) = \frac{\sqrt{-1}}{2\pi} \partial \ol\partial \log \abs{f}_{h^{\vee}}^2, $$ 
where $h^{\vee}$ is the induced singular hermitian metric on the dual line bundle $\mc{L}^{\vee}$
(see \cite[Lemma 18.2]{Schnell18}).  Moreover, the cohomology class of
$\Theta_h(\mc{L})$ in $\hg^2(X, \CC)$ is the first Chern class $c_1(\mc{L})$
(see \cite[\S 13]{Schnell18} and \cite[\S 9.4.D]{robpositivity2}).


\begin{definition}[\protect{\cite[Definition 3.33]{Hor10}} and \protect{\cite[\S 9.3.D]{robpositivity2}}]\label{m-ideals}
	Let $$\varphi \colon \Omega\to [-\infty, +\infty)$$ be a plurisubharmonic 
	function on an open subset $\Omega\subseteq \CC^n$.
	Set 
	\[ \mc{I}(\varphi) := \{ f\in \mc{O}_{\Omega, x}\,|\, \abs{f}^2 e^{-\varphi} \in L_{\text{loc}}^1 \text{ near }x \}. \]
	Let $X$ be a complex manifold, and let $\mc{L}$ be a pseudo-effective line bundle
	endowed with a singular hermitian metric $h$ such that $\Theta_h (\mc{L}) \ge 0$.
	The \emph{(analytic) multiplier ideal sheaf} $\mc{I}(h)$ of $(\mc{L}, h)$ is defined at 
	a point $x\in X$ by $\mc{I}(\varphi)$, where $\varphi$ is a local weight
	function around $x$.  It is well-known that $\mc{I}(h)$ is a 
	coherent sheaf of ideals on $X$.
\end{definition}


\begin{lemma}\label{cur-m-equi-class}
	Let $X$ be a complex manifold.
	Let $(\mc{L}, h)$ and $(\mc{L}', h')$ be two equivalent hermitian $\QQ$-line bundles
	(see Definition-Lemma~\ref{equi-h-metric}).  Then $\mc{I}(h) = \mc{I}(h')$ and
	$\Theta_h(\mc{L}) = \Theta_{h'}(\mc{L}')$.  In other words, for a fixed class 
	$[\mc{L}]\in \picard (X)_{\QQ}$, the curvature and multiplier ideal sheaf 
	of $(\mc{L}, h)$ are uniquely determined 
	by its hermitian class $h_{[\mc{L}]}$.
\end{lemma}

\begin{proof}
	Assume that $(\mc{L}, h) = \{ (U_i, g_{ij}, \varphi_i)\}$ and 
	$(\mc{L}', h') = \{ (U_i, h_{ij}, \psi_i) \}$.  Then there are nowhere vanishing
	holomorphic functions $\rho_i\in \mc{O}_X(U_i)$ such that 
	$\psi_i = \varphi_i + \log \abs{\rho_i}^2$ for every $i$.  Thus, we have
	$\sqrt{-1}\partial\ol\partial \psi_i = \sqrt{-1}\partial\ol\partial \varphi_i$
	for every $i$, and so $\Theta_h(\mc{L}) = \Theta_{h'}(\mc{L}')$ 
	in the sense of currents.  The equality $\mc{I}(h) = \mc{I}(h')$
	follows from the relation $e^{-\psi_i} = e^{-\varphi_i} \cdot \abs{\rho_i}^{-2}$. 
\end{proof}


One of the important properties of analytic multiplier ideal sheaves is the following subadditivity 
property.

\begin{theorem}[\protect{\cite[Theorem 2.6]{subadditivity}}]\label{sub-add}
	Let $X$ be a complex manifold, and let $\phi, \psi$ be plurisubharmonic 
	functions on $X$.  Then
	\[ \mc{I}(\phi + \psi) \subseteq \mc{I}(\phi)\cdot \mc{I}(\psi). \]
	Moreover, if $\psi$ is plurisubharmonic and smooth, then $\mc{I}(\phi + \psi) = \mc{I}(\phi)$.
\end{theorem}


\begin{lemma}\label{psh}
	Let $\varphi \colon \Omega\to [-\infty, +\infty)$ be a plurisubharmonic 
	function on a bounded open subset $\Omega\subseteq \CC^n$ such that $e^{-\varphi}$
	is integrable on $\Omega$.  Then for any $m\in \NN$, 
	$\frac{1}{m} \varphi$ is also a plurisubharmonic function and 
	$e^{-\frac{1}{m}\varphi}$ is integrable.
\end{lemma}

\begin{proof}
    The function $\frac{1}{m}\varphi$ is plurisubharmonic by \cite[Definition 12.1]{Schnell18}.
	Denote by $\varphi^+$ and $\varphi^-$ be the non-negative and negative part of $\varphi$.
	As $\Omega$ is bounded, the integrability of $e^{-\frac{1}{m}\varphi}$
	follows from the facts that $0<e^{-\frac{1}{m}\varphi^+}\le 1$ and 
	$e^{-\frac{1}{m}\varphi^-}< e^{-\varphi^-}$.
\end{proof}


\begin{lemma}\label{fraction-metric}
    Let $X$ be a complex manifold and $\mc{L}$ a 
    $\QQ$-line bundle on $X$.
    Let $m$ be a positive integer such that $\mc{L}^{\otimes m}$ 
    is an integral line bundle.
    Then a singular hermitian metric $H$ of $\mc{L}^{\otimes m}$ with semi-positive curvature
    induces a singular hermitian metric $h$ of $\mc{L}$	
    such that $(\mc{L}, h)$ has semi-positive curvature.
    Moreover, we have: 
    \begin{itemize}
    	\item [\emph{(1)}] $\mc{I}(H)\subseteq \mc{I}(h)$, 
    	\item [\emph{(2)}] for every open subset $U\subseteq X$, 
    	      \[ \abs{s_U^{\otimes m}}_H^2 = \abs{s_U}_h^{2m} \]
    	      for every multi-valued section $s_U$ of $\mc{L}|_U$, and
    	\item [\emph{(3)}] the curvatures satisfy $$\Theta_h(\mc{L}) = \frac{1}{m} \Theta_H(\mc{L}^{\otimes m}). $$
    \end{itemize}
\end{lemma}

\begin{proof}
    Suppose that $\mc{L} = \{(U_i, g_{ij})\}$.
	Then $\{ (U_i, g_{ij}^m) \}$ defines the integral line bundle $\mc{L}^{\otimes m}$.
	Taking refinements of the open covering $\{ U_i \}$ if necessary,
	we can assume that $H$ is given by local weight functions $\{ (U_i, \varphi_i) \}$;
	in particular, we have
	$ e^{-\varphi_i} = \abs{g_{ij}^m}^{-2} e^{-\varphi_j}, $
	which gives
	$ e^{-\frac{1}{m}\varphi_i} = \abs{g_{ij}}^{-2} e^{-\frac{1}{m}\varphi_j}$.
	Thus, $\{ (U_i, \frac{1}{m} \varphi_i) \}$ defines a singular hermitian metric $h$ on $\mc{L}$.
	By Lemma~\ref{psh}, the results on curvature and multiplier ideal sheaves 
	follow immediately from the construction of $h$.
\end{proof}


\begin{lemma}\label{sum-metric}
    Let $X$ be a smooth complex variety.
    Let $M$ and $N$ be pseudo-effective $\QQ$-divisors on $X$ such that
    \begin{itemize}
    	\item there are singular hermitian metrics $h_M$ and $h_N$ of semi-positive curvatures on the associated $\QQ$-line bundles $\mc{L}_M$ and $\mc{L}_N$ respectively, and 
    	\item $M+N$ is an integral divisor with associated integral line bundle $\mc{L}_{M+N}$.
    \end{itemize}
    Then $h_M\cdot h_N$ is a well-defined singular hermitian metric on the integral line bundle
    $\mc{L}_{M+N}$ satisfying
    \[ \Theta_{h_M\cdot h_N}(\mc{L}_{M+N}) = \Theta_{h_M}(\mc{L}_M) + 
    \Theta_{h_N}(\mc{L}_N) \ge 0 \]
    \[ \text{and }\,  \mc{I}(h_M\cdot h_N) \subseteq \mc{I}(h_M)\cdot \mc{I}(h_N). \]
\end{lemma}

\begin{proof}
	We can assume that $\{ (U_i, g_{ij}, \phi_i) \}$ and $\{ (U_i, \lambda_{ij}, \psi_i) \}$
	are the data of the hermitian $\QQ$-line bundles $(\mc{L}_M, h_M)$ and 
	$(\mc{L}_N, h_N)$ respectively.  Taking $m\gg 0$ sufficiently large,
	we can assume that $\mc{L}_M^{\otimes m}$ and $\mc{L}_N^{\otimes m}$ are integral line bundles
	and that $h_M^{\otimes m}$ and $h_N^{\otimes m}$ are singular hermitian
	metrics on the integral line bundles $\mc{L}_M^{\otimes m}$ and $\mc{L}_N^{\otimes m}$ respectively.
	We can also assume that the integral line bundle 
	$\mc{L}_{M+N}$ is given by the data $\{(U_i, \rho_{ij})\}$.
	
	By assumption, the integral line bundles $\mc{L}_{M+N}^{\otimes m}$ 
	and $\mc{L}_M^{\otimes m} \otimes \mc{L}_N^{\otimes m}$
	are isomorphic, and hence the two collections of transition functions
	$\{ (U_i, \rho_{ij}^m) \}$ and $\{ (U_i, g_{ij}^m \lambda_{ij}^m) \}$ are equivalent.
	Then there is a collection of nowhere vanishing holomorphic functions $\{ (U_i, \mu_i) \}$ such that
	\[ g_{ij}^m\lambda_{ij}^m/\rho_{ij}^m = \mu_i \cdot \mu_j^{-1}. \]
	Taking refinements of the open covering $\{U_i\}$ if necessary, we can assume that
	every $U_i$ has trivial topological fundamental group (see \cite[Chapter I, \S 5]{Bott-Tu}), so
	there are holomorphic functions $\beta_i$ on every $U_i$ satisfying $\beta_i^m = \mu_i$.
	Replace the set of transition functions $\{ (U_i, \rho_{ij}) \}$ by $\{ (U_i, \beta_i \rho_{ij} \beta_j^{-1}) \}$.
	Then we have the equality
	\[ e^{-(\phi_i + \psi_i)} = \abs{\rho_{ij}}^{-2} e^{-(\phi_j+\psi_j)} \text{ for all } i,j, \]
	which shows that $\{ (U_i, \phi_i + \psi_i) \}$ gives a well-defined singular hermitian metric
	on $\mc{L}_{M+N}$ (denoted by $h_M\cdot h_N$).  Moreover, as $\sqrt{-1}\partial \ol\partial \phi_i \ge 0$
	and $\sqrt{-1}\partial \ol\partial \psi_i \ge 0$, it is easy to see that
	$\Theta_{h_M\cdot h_N}(\mc{L}_{M+N}) \ge 0$ by Definition~\ref{cur-form}.
	
	Since the local weight functions $\{ (U_i, \phi_i) \}$ and $\{ (U_i, \psi_i) \}$
	are plurisubharmonic, the inclusion of multiplier ideal sheaves $\mc{I}(h_M\cdot h_N) \subseteq \mc{I}(h_M)\cdot \mc{I}(h_N)$
	follows from Theorem~\ref{sub-add}.
\end{proof}


\subsection{Canonical hermitian class of $\QQ$-divisors}\label{cano-h-class-divisors}

Let $X$ be a smooth complex variety (viewed as a complex manifold)
and consider a $\QQ$-divisor $D = \sum_{k=1}^s a_k D_k$
on $X$, where $\{D_k\}$ are the prime components of $D$.
Fix an open covering $\{U_i\}$ of $X$ such that every $D_k$ is principal on every $U_i$.
Let $\alpha_{ki}\in \mc{O}_X(U_i)$ be a holomorphic function locally defining $D_k$ on $U_i$.
We call
\[ \varphi_D (U_i, \alpha_{ki}) := \sum_{k=1}^s a_k \log \abs{\alpha_{ki}}^2 \]
a \emph{local weight function associated to $D$} (with respect to the data $\{(U_i, \alpha_{ki})\}$).

Let $\Xi ([D])$ be the $\QQ$-line bundle class corresponding to $[D]$ (see 
Lemma~\ref{Q-div-Q-line-bdl}).  Pick $m\in \NN(D)$.
Set 
$\alpha_i := \prod_{k=1}^s \alpha_{ki}^{ma_k}$, 
a meromorphic function on $U_i$.
Note that $\alpha_i$ is the defining equation of $m D$ on $U_i$, hence
$\{ (U_i, \alpha_i/\alpha_j) \}$ defines the integral line bundle $\mc{O}_X(mD)$.
Taking $\{U_i\}$ as a good covering (see \cite[page 42]{Bott-Tu}), there are transition functions $g_{ij}\in \mc{O}_X(U_i\cap U_j)$
defining the class $\Xi ([D])$ such that $g_{ij}^m = \alpha_i/\alpha_j$.
Moreover, on the open subset $U_i$, the function $\varphi_D (U_i, \alpha_{ki})$ is given by
$\frac{1}{m}\log \abs{\alpha_i}^2$.  Therefore, we have that
\[ \exp(- \varphi_D (U_i, \alpha_{ki})|_{U_i}) = \abs{\alpha_i}^{-\frac{2}{m}} = 
	\abs{g_{ij}}^{-2} \abs{\alpha_j}^{-\frac{2}{m}} = 
	\abs{g_{ij}}^{-2}\exp(- \varphi_D (U_i, \alpha_{ki})|_{U_j}), \]
which means that
\[ (\mc{L}, h) := \{ (U_i, g_{ij}, \varphi_D (U_i, \alpha_{ki})|_{U_i}) \} \]
is a hermitian $\QQ$-line bundle whose image under $\mc{H}$ is $\Xi ([D])$ 
(see Definition-Lemma~\ref{equi-h-metric}).

On the other hand,
let $\{(U_i, \beta_{ki})\}$ be another set of defining equations of $\{D_k\}$ on $U_i$.
Set 
$\beta_i := \prod_{k=1}^s \beta_{ki}^{ma_k}$ and 
\[ \varphi_D (U_i, \beta_{ki}) := \sum_{k=1}^s a_k \log \abs{\beta_{ki}}^2. \]
Up to refining $\{U_i\}$, pick $\rho_{ij}\in \mc{O}_X(U_i\cap U_j)$ such that
$\rho_{ij}^m = \beta_i/\beta_j$.  Then it is easy to see that
\[ (\mc{L}', h') := \{ (U_i, \rho_{ij}, \varphi_D (U_i, \beta_{ki})|_{U_i}) \} \]
is a hermitian $\QQ$-line bundle equivalent to $(\mc{L}, h)$.
Therefore, the construction above gives a unique hermitian class of $\Xi ([D])$,
which we call the \emph{canonical hermitian class of $D$}.


\begin{remark}\label{eff-Q-div-h-metric}
    Let $X$ be a smooth complex variety and let
	$D = \sum_{k=1}^s a_k D_k$ be an effective $\QQ$-divisor on $X$.
	Let $(\mc{L}_D, h_D)$ be a hermitian $\QQ$-line bundle 
	representing the canonical hermitian class of $D$.
	Then $(\mc{L}_D, h_D)$ has semi-positive curvature.
    Indeed, this follows immediately from the construction above as
	the local weight function is plurisubharmonic (since $a_k\ge 0$ for every $1\le k\le s$).
\end{remark}


\subsection{Pullbacks and restrictions of singular hermitian metrics}

\begin{lemmadef}\label{pullback-curvature}
	Let $f\colon Y\to X$ be a morphism of complex manifolds and $(\mc{L}, h)$
	a hermitian $\QQ$-line bundle on $X$ given by the data
	$\{ (U_i, g_{ij}, \varphi_i) \}$, where $\varphi = \{ (U_i, \varphi_i) \}$ is
	the local weight functions of $h$.  Then the pullback data 
	$$\{ (f^{-1}(U_i), g_{ij}\circ f, \varphi_i\circ f) \}$$
	defines a hermitian $\QQ$-line bundle on $Y$, which we denote
	by $( f^*\mc{L}, f^*h )$.  If $f$ is a closed embedding, we write 
	$( f^*\mc{L}, f^*h )$ as $( \mc{L}|_Y, h|_Y )$.
	
	Moreover, we define the \emph{pullback curvature}
	$f^{\star}\Theta_h(\mc{L})$ to be the current locally given by
	\[ \frac{\sqrt{-1}}{2\pi}\partial\ol\partial (\varphi_i\circ f) \text{ on every } U_i. \]
	Then we have the equality
	\[ \Theta_{f^*h}(f^* \mc{L}) = f^{\star} \Theta_h(\mc{L}).\]
	Furthermore, if $(\mc{L}, h)$ has semi-positive curvature, then $( f^*\mc{L}, f^*h )$
	has semi-positive curvature too.
\end{lemmadef}

\begin{proof}
	It is clear that $\{ (f^{-1}(U_i), g_{ij}\circ f, \varphi_i\circ f) \}$
	is a hermitian $\QQ$-line bundle on $Y$.  The equality of curvatures follows
	immediately from the definition of $f^{\star}\Theta_h(\mc{L})$.
	Moreover, if $(\mc{L}, h)$ has semi-positive curvature, then $\varphi_i$
	is plurisubharmonic, hence $\varphi_i\circ f$ is also plurisubharmonic 
	(see \cite[Corollary 2.9.5]{pluri-p-theory}), which is equivalent to saying that
	$(f^*\mc{L}, f^* h)$ has semi-positive curvature.
\end{proof}
	

\begin{remark}
    Even in the setting of Definition-Lemma~\ref{pullback-curvature}, it is quite subtle
    to pull back a general current (or a distribution) from $X$ to $Y$, 
    see \cite[Chapters VI, VIII]{Hormander}.  However, the definition given in Definition-Lemma~\ref{pullback-curvature}
    is sufficient for the arguments in this paper.
    To avoid confusion with the notions given in \cite{Hormander}, we adopt the notation
    $f^{\star}\Theta_h(\mc{L})$, instead of $f^*\Theta_h(\mc{L})$, for the pullback curvature.
\end{remark}


\begin{lemma}\label{res-pullback-cano-metric}
    Let $f\colon Y\to X$ be a surjective morphism of smooth complex varieties, and $B$
    a $\QQ$-divisor on $Y$ that is snc over the generic point of $X$.
    Denote by $(\mc{L}, h)$ a hermitian $\QQ$-line bundle 
    representing the canonical hermitian class of $B$.
    Let $F$ be a general fibre of $f$.  Then the restriction $(\mc{L}|_F, h|_F)$
    is a hermitian $\QQ$-line bundle 
    representing the canonical hermitian class of the $\QQ$-divisor $B|_F$.
    
    Similarly, if $D$ is a $\QQ$-divisor on $X$
    and if $(\mc{L}_D, h_D)$ is a hermitian $\QQ$-line bundle representing 
    the canonical hermitian class of $D$, then
    $(f^*\mc{L}_D, f^* h_D)$ represents the canonical hermitian class of $f^* D$.
\end{lemma}

\begin{proof}
    Shrinking $X$ if necessary, we can assume that $B$ is an snc $\QQ$-divisor.
	Suppose that $B = \sum_{i=1}^s a_k B_k$ and that every $B_k$ has local defining
	equation $g_k$ on some open subset $U\subseteq Y$. 
	By generic smoothness, every connected component of $B_k|_F$ is a smooth prime divisor on $F$,
	and $B|_F = \sum_{i=1}^s a_k B_k|_F$ is an snc $\QQ$-divisor on $F$ 
	whose prime components have local defining equations $\{g_{k,F}\}$ on $U\cap F$.
	Here, $g_{k,F}$ denotes the image of $g_k$ under the morphism $\mc{O}_U\to \mc{O}_{U\cap F}$,
	which is a holomorphic function on $U\cap F$.
	So, the local weight functions 
	$$\varphi_{F} := \sum_{k=1}^s a_k \log\abs{g_{k,F}}^2 $$
	induce a hermitian $\QQ$-line bundle $(\mc{L}|_F, h_F)$ which
	represents the canonical hermitian class of $B|_F$ (see \S\ref{cano-h-class-divisors}).
	On the other hand, recall that the local weight functions of $(\mc{L}, h)$
	are given by 
	$\varphi_B := \sum_{k=1}^s a_k \log \abs{g_k}^2$, up to equivalence of $(\mc{L}, h)$.
	As $\varphi_F$ is the restriction of $\varphi_B$ to $F$, we conclude that
	the restricted hermitian metric $h|_F$ is equal to $h_F$.
	The assertion about pullback divisor and pullback metric 
	follows in the same way.
\end{proof}


\subsection{Singular hermitian metrics of vector bundles}

The notion of \emph{semi-positive curvature} of a hermitian vector bundle 
$(\mc{E}, h)$, a holomorphic vector bundle $\mc{E}$ endowed with a singular hermitian metric
on a complex manifold,
is similarly defined as for hermitian (integral) line bundles; 
we refer the readers to \cite[\S 2.3]{PT18} and \cite[\S 18]{Schnell18} for relevant definitions.
More generally, singular hermitian metrics and curvatures can also 
be defined for torsion free sheaves; for details, see \cite[\S 2.4]{PT18}.


Let $(\mc{E}, h)$ be a hermitian vector bundle on a complex manifold $X$.
Then $h$ induces a singular hermitian metric $h^{\vee}$ on the dual vector bundle $\mc{E}^{\vee} := \shomo_{\mc{O}_X} (\mc{E}, \mc{O}_X)$
(see \cite[Proposition 17.2]{Schnell18}).
Let $\mc{Q}$ be a quotient vector bundle of $\mc{E}$, i.e., there exists a surjective morphism
of (holomorphic) vector bundles $q\colon \mc{E}\to \mc{Q}$.
By dualising the morphism $q$, we obtain an injection of vector bundles $q^{\vee} \colon \mc{Q}^{\vee}\to \mc{E}^{\vee}$.
Then the restriction $h^{\vee}|_{\mc{Q}^{\vee}}$ is a singular hermitian metric on $\mc{Q}^{\vee}$, and so
$h_{\mc{Q}} := (h^{\vee}|_{\mc{Q}^{\vee}})^{\vee}$ is a singular hermitian metric on $\mc{Q}\simeq \mc{Q}^{\vee \vee}$.
As $\ker q$ is also a (holomorphic) vector bundle on $X$, 
the metric $h_{\mc{Q}}$ defined above is the same as the hermitian metric 
on quotient vector bundles as in \cite[page 78]{Griffiths-Harris}.
Moreover, if $(\mc{E}, h)$ has semi-positive curvature, then $(\mc{Q}, h_{\mc{Q}})$ also has 
semi-positive curvature (see \cite[Lemma 2.4.3]{PT18} and \cite[Lemma 18.2]{Schnell18}).


\subsection{Positivity of direct images}

Let $f\colon X\to Y$ be a contraction of two complex manifolds.
Let $\mc{L}$ be an integral line bundle on $X$ endowed with 
a singular hermitian metric $h$ with semi-positive curvature.
Denote by $Y_0\subset Y$ the maximal Zariski open subset over which $f$ is smooth
and $f_*(\omega_{X/Y}\otimes \mc{L})$ is locally free.
Then $f_*(\omega_{X/Y}\otimes \mc{L})|_{Y_0}$ admits a 
\emph{canonical $L^2$-metric with respect to $h$}.  
Moreover, this canonical $L^2$-metric extends as a singular 
hermitian metric $h_{X/Y}$ on the torsion free sheaf $f_*(\omega_{X/Y}\otimes \mc{L})$
with semi-positive curvature.  We say that $h_{X/Y}$ is the \emph{$L^2$-metric
on $f_*(\omega_{X/Y}\otimes \mc{L})$ with respect to $h$}.
The technical explicit construction of $h_{X/Y}$ will not be used in this paper, 
so we refer the readers to \cite[\S 3.2]{PT18} for more details.


\begin{proposition}[cf. \protect{\cite[Lemma 5.4]{CP17}}]\label{pushf-positivity}
	Let $p\colon X\to Y$ be a contraction of complex manifolds.
	Let $(\mc{G}, h_{\mc{G}})$ and $(\mc{L}, h_{\mc{L}})$ be hermitian integral line bundles 
	on $X$ and $Y$ respectively.  Denote by $\omega_{X/Y}$
	the relative dualising sheaf of $f$.
	We assume that the following requirements hold true.
	\begin{itemize}
		\item [\emph{(i)}] We have $\Theta_{h_{\mc{L}}}(\mc{L}) \ge 0$ and 
		      \[ \Theta_{h_{\mc{G}}}(\mc{G}) \ge \varepsilon_0 p^{\star}\Theta_{h_{\mc{L}}}(\mc{L}) \text{ on }X, \]
		      where $\varepsilon_0$ is a positive real number.
		\item [\emph{(ii)}] The direct image $p_*(\omega_{X/Y} \otimes \mc{G})$ is nonzero, and we have
		      \[ p_*(\omega_{X/Y}\otimes \mc{G} \otimes \mc{I}(h_{\mc{G}}))_y = p_*(\omega_{X/Y}\otimes \mc{G})_y \]
		      for a general closed point $y\in Y$.
	\end{itemize}
	Let $h_{X/Y}$ be the $L^2$-metric on the direct image sheaf $p_*(\omega_{X/Y}\otimes \mc{G})$ 
	with respect to $h_{\mc{G}}$.  Then for any vector bundle $\mc{E}$
	equipped with a surjection of sheaves
	\[ p_*(\omega_{X/Y}\otimes \mc{G})\to \mc{E}\to 0 \] 
	we have
	\[ \Theta_{\det h_{\mc{E}}} (\det \mc{E}) \ge \varepsilon_0 \Theta_{h_{\mc{L}}}(\mc{L}) \text{ on } Y, \]
	where $h_{\mc{E}}$ is the quotient singular hermitian metric
	induced by $h_{X/Y}$.
\end{proposition}

\begin{proof}
	The proof is almost the same as the one of \cite[Lemma 5.4]{CP17}; we include the proof here
	for the convenience of the readers and (in order) to correct some typos in the proof of \cite[Lemma 5.4]{CP17}.\\
	
	\emph{Step 1}.
	Let $y_0\in Y$ be an arbitrary point, and let $\Omega_0\subset Y$ be an open subset 
	(in the analytic topology of complex manifolds) of $Y$ centred at $y_0$.
	Let $\Omega\subset p^{-1}(\Omega_0)$ be an open subset of $X$ contained in the 
	inverse image of $\Omega_0$.  Shrinking $\Omega$ if necessary, we can assume 
	that $h_{\mc{G}}$ on $\Omega$ is given by the local weight function $\varphi_{\mc{G}}$,
	and that $h_{\mc{L}}$ on $\Omega_0$ is given by $\varphi_{\mc{L}}$.
	By the conditions in (i), we have 
	\[ \sqrt{-1}\partial\ol\partial \varphi_{\mc{L}} \ge 0 \text{ and }
	\sqrt{-1}\partial\ol\partial \varphi_{\mc{G}}\ge \varepsilon_0 p^{\star}  \sqrt{-1} \partial\ol\partial \varphi_{\mc{L}} \text{ on } \Omega.\]
	We can write the local weight function $\varphi_{\mc{G}}$ of $h_{\mc{G}}$ as
	\[ \varphi_{\mc{G}} = (\varphi_{\mc{G}} - \varepsilon_0 \varphi_{\mc{L}} \circ p) + \varepsilon_0 \varphi_{\mc{L}} \circ p \]
	which is a sum of two plurisubharmonic functions (up to modification on a set of
	measure zero).
	
	By varying the open subset $\Omega$ inside $p^{-1}(\Omega_0)$, the local expressions
	$\varphi_{\mc{G}} - \varepsilon_0 \varphi_{\mc{L}} \circ p$ glue together as 
	local weight functions of a singular hermitian metric on the line bundle 
	$\mc{G}|_{p^{-1}(\Omega_0)}$, which we denote by
	\[ h_{0\mc{G}} := e^{\varepsilon_0 \varphi_{\mc{L}} \circ p} \cdot h_{\mc{G}}. \]
	As $\varphi_{\mc{G}} - \varepsilon_0 \varphi_{\mc{L}} \circ p$ is plurisubharmonic,
	we have that
	\[ \Theta_{h_{0\mc{G}}} (\mc{G}|_{p^{-1}(\Omega_0)}) \ge 0. \]
	Notice that condition (ii) is also satisfied by $h_{0\mc{G}}$.  
	Denote by $H_{X/Y}$ the $L^2$-metric with respect to $h_{0\mc{G}}$ 
	on $p_*(\omega_{X/Y}\otimes \mc{G})|_{\Omega_0}$.	Then
	\cite[Theorem 3.3.5]{PT18} shows that 
	\[ (p_*(\omega_{X/Y}\otimes \mc{G})|_{\Omega_0}, H_{X/Y}) \]
	has semi-positive curvature.\\
	
	\emph{Step 2}.
	Now, let $q\colon p_*(\omega_{X/Y}\otimes \mc{G}) \to \mc{E}$ be a quotient vector bundle of 
	rank $r$ of $p_*(\omega_{X/Y}\otimes \mc{G})$, and let
	$q_0$ be the restriction of $q$ on the open subset $\Omega_0$.
	We have the inclusion of dual sheaves
	\[ 0\to \mc{E}^{\vee}|_{\Omega_0} \to p_*(\omega_{X/Y}\otimes \mc{G})^{\vee}|_{\Omega_0}. \]
	Fix a local holomorphic frame $\sigma = \{ \sigma_1, \dots, \sigma_r \}$ of $\mc{E}^{\vee}|_{\Omega_0}$
	(up to shrinking $\Omega_0$ around $y_0\in Y$ if necessary).
	Then every local holomorphic section of $\mc{E}^{\vee}|_{\Omega_0}$ is of the form $\xi_1 \sigma_1 + \cdots + \xi_r \sigma_r$
	where every $\xi_i$ is a holomorphic function under the trivialisation
	of $\mc{E}^{\vee}|_{\Omega_0}$ induced by $\sigma$.
	Thus, every local holomorphic section of $(\det \mc{E})^{\vee}|_{\Omega_0}$
	can be identified with the holomorphic function $\xi = \xi_1\cdots \xi_r$.
	
	Denote by $H_{\mc{E}}$ (resp. $h_{\mc{E}}$) the quotient singular hermitian metric on $\mc{E}|_{\Omega_0}$
	induced by $H_{X/Y}$ (resp. by $h_{X/Y}$).
	Let $\mc{L}_i$ be the quotient line bundle $\mc{E}|_{\Omega_0}\to \mc{L}_i$ 
	induced by $\xi_i\sigma_i$.  Then we have the inclusions
	$\mc{L}_i^{\vee} \subseteq \mc{E}^{\vee}|_{\Omega_0}$ and 
	the quotients of sheaves
	\[ p_*(\omega_{X/Y}\otimes \mc{G})|_{\Omega_0} \to \mc{L}_i\to 0. \]
	By \cite[Lemma 18.2]{Schnell18}, $\mc{L}_i$ has semi-positive curvature 
	when endowed with the quotient singular hermitian metric induced by $H_{\mc{E}}$
	(which is the same as the quotient singular hermitian metric induced from $(p_*(\omega_{X/Y}\otimes \mc{G})|_{\Omega_0}, H_{X/Y})$).
	In other words, we have
	\[ \sqrt{-1} \partial \ol\partial \log \abs{\xi_i}_{H_{\mc{E}}^{\vee}}^2 = \sqrt{-1}\partial \ol\partial \log \abs{\xi_i}_{H_{X/Y}^{\vee}}^2 \ge 0,  \]
	which implies that 
	\[ \sqrt{-1}\partial \ol\partial \log \abs{\xi_i}^2_{h_{\mc{E}}^{\vee}} = \sqrt{-1}\partial \ol\partial \log \abs{\xi_i}_{h_{X/Y}^{\vee}}^2 \ge \varepsilon_0 \sqrt{-1} \partial \ol\partial \varphi_{\mc{L}} 
	\text{ for every }1\le i\le r. \]
	Therefore, we have 
	\[ \sqrt{-1}\partial \ol\partial \log \abs{\xi}^2_{\det h_{\mc{E}}^{\vee}} \ge r \varepsilon_0 \sqrt{-1} \partial \ol\partial \varphi_{\mc{L}}
	\ge \varepsilon_0 \sqrt{-1} \partial \ol\partial \varphi_{\mc{L}}, \]
	which is equivalent to saying that $\Theta_{\det h_{\mc{E}}} (\det \mc{E}) \ge \varepsilon_0 \Theta_{h_{\mc{L}}}(\mc{L})$ on $Y$.
\end{proof}


\section{Injectivity theorem for g-pairs with b-nef-and-abundant nef part}\label{b-abnd-results-sec}

The following result, its proof, and corollaries will be used
in the proof of Theorem~\ref{inj-surface-b-iitaka}.

\begin{theorem}\label{b-kodaira-non-negative}
    Suppose that $(X, B+\mathbf{M})$ is a projective g-klt $\QQ$-g-pair. 
	Assume that $L$ is an integral $\mathbb{Q}$-Cartier
	divisor such that
	$A = L- (K_X+B+ \mathbf{M}_X)$
	is semi-ample.  Denote by $\mathbf{A}$ the $\QQ$-b-divisor associated to $A$.
	Assume that one of the following conditions holds:
	\begin{itemize}
		\item [\emph{(1)}] $\mathbf{A} + \mathbf{M}$ is b-nef-and-abundant and $\mathbf{M}$ has b-Iitaka dimension $\ge 0$.
		\item [\emph{(2)}] $\mathbf{A} + \mathbf{M}$ descends to a nef and big divisor on some birational model of $X$.
	\end{itemize}
	Then the injectivity theorem holds for $\{(X, B + \mathbf{M}), L\}$.
\end{theorem}

\begin{proof}
By Proposition~\ref{reduce-to-cartier}, we can assume that $L$ is an integral Cartier divisor.
Moreover, by Lemma~\ref{suff-large-m}, it suffices to prove that the injectivity theorem holds
for $\{ (X, B + \mathbf{M}), L, D \}$ where $D$ is an effective Cartier divisor
in the linear system $\abs{mA}$ for sufficiently large $m\gg 0$.\\

\emph{Step 1}.
Assume that $\phi\colon Y\to X$ is a log resolution for the couple 
$$ (X, \floor{\supp B \cup \supp \mathbf{M}_X}) $$
such that the $\QQ$-b-divisor $\mathbf{M}$ descends to a nef 
$\QQ$-divisor $\mathbf{M}_Y$ on $Y$ (in particular, $\phi_* \mathbf{M}_Y = \mathbf{M}_X$).
Write
\begin{equation}
	K_Y + B_Y + \mathbf{M}_Y =\phi^*(K_X +B + \mathbf{M}_X).\label{discrep-equ-1}
\end{equation}
Rewrite \eqref{discrep-equ-1} as
\[ K_Y+\mathbf{M}_Y+ \phi^*(L-(K_X +B+\mathbf{M}_X)) = \phi^* L -B_Y, \]
and so 
\begin{equation}
	\ceil{\phi^*L - B_Y} = K_Y + \ceil{ \mathbf{M}_Y + \phi^*(L-(K_X +B+\mathbf{M}_X)) },\label{discrep-equ-2}
\end{equation}
as $K_Y$ is integral.
Since $(X, B+\mathbf{M})$ is g-klt, we have $-B_Y>-1$, 
which implies that $\ceil{-B_Y} \ge 0$ and $\ceil{-B_Y}$
is $\phi$-exceptional as $B$ is effective.  
The $\QQ$-divisor
\begin{equation}
 G_Y := \mathbf{M}_Y + \phi^*(L-(K_X +B+\mathbf{M}_X)) = 
 \mathbf{M}_Y + \phi^* A = \mathbf{M}_Y + \mathbf{A}_Y \label{GY-defn}
\end{equation}
is nef and big over $X$
since $\mathbf{M}_Y$ is nef$/X$ and since $\phi^*A$ is also semi-ample.
Thus, by the Kawamata-Viehweg vanishing theorem (see \cite[Theorem 1-2-3]{kmm}), we have 
\[ R^i\phi_*\mc{O}_Y(K_Y + \ceil{ G_Y})=0 \text{ for all } i>0.\]
Then the Leray spectral sequence gives that
\[ \hg^i(Y, \mc{O}_Y( K_Y +\ceil{G_Y})) = \hg^i(Y, \mc{O}_Y(\ceil{\phi^*L-B_Y}))
= \hg^i(X, \phi_* \mc{O}_X (\ceil{\phi^*L-B_Y})). \]
As $\phi^* L$ is Cartier, we have 
$$\ceil{\phi^* L - B_Y} = \phi^* L + \ceil{-B_Y}.$$
However, since the effective Cartier divisor $\ceil{ - B_Y}$ is $\phi$-exceptional,
we have 
$$\phi_*\mc{O}_Y(\phi^*L + \ceil{-B_Y}) = \phi_*\mc{O}_Y(\phi^*L) = \mc{O}_X(L)$$
by \cite[Lemma 7.11]{Debarre_higher_dim_ag} and \cite[Lemma II.2.11]{nakayama}.  
Therefore, we get the isomorphism
\begin{equation}
	\hg^i(Y, \mc{O}_Y(K_Y +\ceil{G_Y})) = \hg^i(X, \mc{O}_X(L)) 
	\text{ for all } i\ge 0.\label{iso-ch-1}
\end{equation}
Note that $G_Y + \phi^* D$
is also nef and big over $X$.  Thus, by the same argument as above, 
we also have the isomorphisms of cohomologies
\begin{equation}
	\hg^i(Y, \mc{O}_Y(K_Y +\ceil{ G_Y + \phi^*D}))
=\hg^i(X, \mc{O}_X(L + D)) \text{ for all } i\ge 0.\label{iso-ch-2}
\end{equation}
Also, note that $K_Y + \ceil{G_Y + \phi^* D}$ is equal to $K_Y + \phi^*D + \ceil{G_Y} $
since $D$ is a Cartier divisor.\\

\emph{Step 2}.
The $\QQ$-b-divisor $\mathbf{A} + \mathbf{M}$ descends to $Y$, and 
\[ \mathbf{A}_Y + \mathbf{M}_Y = \phi^*( L- (K_X + B +\mathbf{M}_X)) + \mathbf{M}_Y \]
is nef and abundant by assumptions (up to taking higher resolutions),
see \eqref{GY-defn}.
Denote by $\Delta$ the $\QQ$-divisor:
\[ \ceil{ \mathbf{M}_Y + \phi^*( L- (K_X + B +\mathbf{M}_X))} 
- (\mathbf{M}_Y + \phi^*( L- (K_X + B +\mathbf{M}_X))) \]
which by construction is a simple normal crossing $\QQ$-divisor of the form 
$\Delta = \sum_k \eta_k \Delta_k$ with $0<\eta_k < 1$ for all $k$.
Set
\begin{equation}
 \Gamma_Y := \ceil{G_Y} = \ceil{\mathbf{M}_Y + \phi^*( L- (K_X + B +\mathbf{M}_X))} = \ceil{\mathbf{M}_Y + \mathbf{A}_Y}, \label{GammaY-defn} 
\end{equation}
where $G_Y$ is defined in \eqref{GY-defn}.  We have
\[ \Gamma_Y - \Delta = \mathbf{M}_Y + \phi^*( L- (K_X + B +\mathbf{M}_X)) 
= \mathbf{M}_Y + \mathbf{A}_Y. \]

In case (1), by taking $Y$ as a high enough resolution, we can assume that 
$\mathbf{M}_Y + \mathbf{A}_Y$ is nef and abundant and that $\kappa (\mathbf{M}_Y)\ge 0$.
In particular, $\mathbf{M}_Y$ is $\QQ$-linearly equivalent to some effective divisor.
Recall that $D$ is an effective Cartier divisor in the linear system
$\abs{mA}$ for sufficiently large $m\gg 0$.
Thus, by taking a rational number $0< \varepsilon< 1$ sufficiently small, the divisor
\[ \Gamma_Y - \Delta - \varepsilon \phi^* D = \mathbf{M}_Y 
  + (\mathbf{A}_Y - \varepsilon \phi^* D)  \]
is $\QQ$-effective.  

In case (2), we can assume that $\mathbf{M}_Y + \mathbf{A}_Y$ is nef and big.  Thus,
there is an effective divisor $N$ on $Y$ such that for every $t\gg 0$ we can write
\[ \mathbf{M}_Y + \mathbf{A}_Y \sim_{\QQ} H_t + \frac{1}{t}N, \]
where $H_t$ is an ample $\QQ$-divisor.
Recall that $D\in \abs{mA}$ for some $m\gg 0$.
Taking a rational number $0<\varepsilon<1$ sufficiently small, $H_t - \varepsilon \phi^* D$
is also ample.  Then we also conclude that
\[ \Gamma_Y - \Delta - \varepsilon \phi^* D = \mathbf{M}_Y + \mathbf{A}_Y -\varepsilon \phi^* D \sim_{\QQ} (H_t - \varepsilon\phi^* D) + \frac{1}{t}N \]
is $\QQ$-effective.

Moreover, with the assumptions in (1) or (2), $\Gamma_Y - \Delta = \mathbf{M}_Y + \mathbf{A}_Y$
is nef and abundant.  Consequently,
we are now in the situation of Theorem~\ref{Ein:Popa}, which implies that
\[ \hg^i(Y, \mc{O}_Y(K_Y + \Gamma_Y)) \to \hg^i(Y, \mc{O}_Y(K_Y + \Gamma_Y + \phi^*D)) 
\text{ \,\,are injective for all }i\ge 0. \]
By \eqref{iso-ch-1}, \eqref{iso-ch-2}, and \eqref{GammaY-defn}, 
this is equivalent to saying that
\[ \hg^i(X, \mc{O}_X(L))\to \hg^i(X, \mc{O}_X(L + D)) \text{ \,\,are injective for all }i\ge 0. \qedhere \]
\end{proof}


\begin{corollary}[=Proposition~\ref{abundant-inj-g-pair-intro}]\label{abundant-inj-g-pair}
    Suppose that $(X, B+\mathbf{M})$ is a projective g-klt $\QQ$-g-pair. 
    Assume that $\mathbf{M}$ is b-nef-and-abundant.
    Then the injectivity theorem holds for $(X, B + \mathbf{M})$.
\end{corollary}

\begin{proof}
    Assume that $L$ is an integral $\mathbb{Q}$-Cartier
	divisor such that
	$A = L- (K_X+B+ \mathbf{M}_X)$
	is semi-ample.  Denote by $\mathbf{A}$ the $\QQ$-b-divisor associated to $A$.
	Since $\mathbf{M}$ is b-nef-and-abundant, $\mathbf{A} + \mathbf{M}$ is also b-nef-and-abundant
	by Lemma~\ref{pre:nef:abundant} (ii).
	Moreover, suppose that $\mathbf{M}$ descends to a nef and abundant $\QQ$-divisor
	$\mathbf{M}_Y$ on a birational model $Y$ of $X$.
	As $\mathbf{M}_Y$ is abundant, $\kappa (\mathbf{M}_Y) \ge 0$.
	Thus, the injectivity theorem holds for $\{ (X, B + \mathbf{M}), L \}$ 
	by Theorem~\ref{b-kodaira-non-negative} (1).
\end{proof}


Note that Corollary~\ref{abundant-inj-g-pair} could also be derived quickly by varying
the g-klt $\QQ$-g-pair $(X, B + \mathbf{M})$ to a usual klt pair 
by Lemma~\ref{abundant-to-nef-big}.
Another immediate corollary of Theorem~\ref{b-kodaira-non-negative} 
is that the injectivity theorem holds 
if the semi-ample divisor is sufficiently positive.

\begin{corollary}\label{nef-big-inj-g-pair}
    Suppose that $(X, B+\mathbf{M})$ is a projective g-klt $\QQ$-g-pair.
    Assume that $L$ is an integral $\mathbb{Q}$-Cartier
	divisor such that
	$A = L- (K_X+B+\mathbf{M}_X)$ is a big and semi-ample $\QQ$-divisor.
	Then the injectivity theorem holds for $\{(X, B + M), L\}$.
\end{corollary}

\begin{proof}
    Denote by $\mathbf{A}$ the $\QQ$-b-divisor associated to $A$.
    Since $A$ is nef and big, $\mathbf{A} + \mathbf{M}$ descends to 
    a nef and big divisor on some birational model of $X$ 
    (see \cite[Theorem 2.2.16]{robpositivity1}), so 
    we can apply Theorem~\ref{b-kodaira-non-negative} (2) to conclude.
\end{proof}


\section{Injectivity theorem for surfaces}\label{inj-general-surfaces}

First, we show that the assumptions in Theorem~\ref{b-kodaira-non-negative}, such as
the b-nef-and-abundant condition on nef parts,
can be weakened for generalised pairs on projective normal surfaces.

\begin{corollary}\label{inj-surface}
    Let $(X, B+ \mathbf{M})$ be a g-klt $\QQ$-g-pair
    on a projective normal surface $X$.
    Assume that the nef part $\mathbf{M}$ has non-negative b-Iitaka dimension.
    Then the injectivity theorem holds for $(X, B+\mathbf{M})$.
\end{corollary}

\begin{proof}
    Take an integral $\QQ$-Cartier divisor $L$ such that 
    $A = L - (K_X + B + \mathbf{M}_X)$ is semi-ample.
	By taking sufficiently high log resolutions, we can 
	take a smooth birational model $Y\xrightarrow{\phi} X$ on which
	$\mathbf{M}$ descends to be a nef $\QQ$-divisor $\mathbf{M}_Y$ satisfying
	$\kappa (\mathbf{M}_Y) \ge 0$.
	Denote by $\mathbf{A}$ the $\QQ$-b-divisor associated to $A$.
	By Theorem~\ref{b-kodaira-non-negative} (1), it suffices to show that 
	$\mathbf{A}_Y + \mathbf{M}_Y$ is nef and abundant.
	Since $\mathbf{A}_Y$ is nef and abundant, 
	$\kappa (\mathbf{A}_Y) = \nu (\mathbf{A}_Y) \ge 1$ (cf. Lemma~\ref{red-kod-dim-not-zero}), hence 
	$$\kappa (\mathbf{A}_Y + \mathbf{M}_Y)\ge \kappa (\mathbf{A}_Y) \ge 1.$$
	If $\kappa (\mathbf{A}_Y + \mathbf{M}_Y) = 2$, then also $\nu (\mathbf{A}_Y + \mathbf{M}_Y) = 2$ by 
	\cite[Chapter V, Lemma 2.1]{nakayama}, that is, $\mathbf{A}_Y + \mathbf{M}_Y$ is nef and abundant.
	If $\kappa (\mathbf{A}_Y + \mathbf{M}_Y) = 1$, then by the same lemma from \cite{nakayama},
	$\nu(\mathbf{A}_Y + \mathbf{M}_Y)$ can not be 2, and so also $\nu (\mathbf{A}_Y + \mathbf{M}_Y) =1$.
	As a summary, $\mathbf{A}_Y + \mathbf{M}_Y$ is always nef and abundant 
	under the condition $\kappa (\mathbf{M}_Y)\ge 0$.
\end{proof}

\begin{remark}
    Nef divisors can have Iitaka dimension $-\infty$. 
    For example, pick a line bundle $\mc{L}\in \picard^0(X)$ on a nonsingular
    projective variety $X$,
    that is, $\mc{L}$ is a numerically trivial line bundle (see \cite[Corollary 1.4.38]{robpositivity1}).
    Write $\mc{L} = \mc{O}_X(D)$  
    for a Cartier divisor $D$.
    Then $\kappa (D) = 0$ if $\mc{L}$ is trivial or torsion,
    but $\kappa (D) = -\infty$ otherwise (see \cite[Example 2.1.9]{robpositivity1}).
\end{remark}


Now, we show that the injectivity theorem holds 
for g-klt $\QQ$-g-pairs on projective surfaces without any extra conditions
on the nef parts of the generalised pairs.

\begin{theorem}[=Theorem~\ref{inj-surface-b-iitaka}]
    Let $X$ be a normal projective variety of dimension 2, and
    let $(X, B+ \mathbf{M})$ be a projective g-klt $\QQ$-g-pair. 
    Then the injectivity theorem holds for $(X, B+\mathbf{M})$.
\end{theorem}

\begin{proof}
By Proposition~\ref{reduce-to-cartier}, it suffices to consider the case when $L$
is an integral \emph{Cartier} divisor such that $L- (K_X + B + \mathbf{M}_X)$ is semi-ample.
Moreover, by Lemma~\ref{suff-large-m}, we can assume that $D$ is an effective Cartier divisor
in the linear system $\abs{m(L- (K_X + B + \mathbf{M}_X))}$ for sufficiently large $m\gg 0$.\\

\emph{Step 1}.
Assume that the given g-klt $\QQ$-g-pair $(X, B+\mathbf{M})$ is 
given by the data $(Y\xrightarrow{\phi} X, B_Y + \mathbf{M}_Y)$, where
$\phi\colon Y\to X$ is a birational morphism from a 
smooth projective surface $Y$ such that $\phi$ is a log resolution of $(X, B + \mathbf{M}_X)$ and
$\mathbf{M}_Y$ is a nef $\QQ$-divisor on $Y$ satisfying $\mathbf{M} = \ol{\mathbf{M}_Y}$.
By \cite[Remark 1.15]{Filipazzi18} and \cite[Remark 4.2 (3)]{BirkarZhang}
(see also \cite[Lemma 2.4]{Han-Liu}) we deduce that $(X, B)$
is a klt pair and that $\mathbf{M}_X$ is a nef $\QQ$-divisor on $X$.
By considering a small $\QQ$-factorial modification of $(X, B)$
\cite[Corollary 1.37]{Kollar-sing-MMP-book} and using the fact 
that there is no small contraction between surfaces, we conclude
that $X$ is $\QQ$-factorial.
Moreover, by considering a $\QQ$-factorial terminal modification of
the $\QQ$-factorial g-klt $\QQ$-g-pair $(X, B + \mathbf{M})$ 
and using the fact that a terminal surface is smooth,
we can assume that $\phi\colon Y\to X$ is a birational contraction from a 
\emph{smooth} projective surface $Y$
and that $\mathbf{M}_Y$ is a \emph{nef} $\QQ$-divisor (see \cite[Remark 1.15, Lemma 3.4, Remark 3.6]{Filipazzi18}) on $Y$ such that
\[ K_Y + B_Y + \mathbf{M}_Y = \phi^*(K_X + B + \mathbf{M}_X) \text{ with } 0\le B_Y < 1. \]
By construction, $(Y, B_Y + \mathbf{M})$ is also a g-klt $\QQ$-g-pair.\\


\emph{Step 2}.
Let $A$ be the semi-ample $\QQ$-divisor on $X$ given by
$$A := L - (K_X + B + \mathbf{M}_X).$$
Denote by $\mathbf{A}$
the $\QQ$-b-divisor associated to $A$.
Denote by $f\colon X\to Z$ the contraction 
induced by the semi-ample divisor $A$.  
Then $A = f^* H_Z$ for some ample $\QQ$-divisor $H_Z$ on $Z$.

\emph{Step 2a}.
If $f$ is birational, then the injectivity theorem holds for 
$\{(X, B + \mathbf{M}), L\}$
by Corollary~\ref{nef-big-inj-g-pair} (as $A$ is nef and big).

\emph{Step 2b}.
If $f$ is not birational, then we can assume that $Z$ is a smooth projective curve (cf. Lemma~\ref{red-kod-dim-not-zero}).
Observe that $\mathbf{A}_Y = \phi^* A$ is also semi-ample.
The induced morphism $g:= f\circ \phi\colon Y\to Z$ is the contraction associated to $\mathbf{A}_Y$.
\[\xymatrix{
Y\ar[r]^{\phi}\ar[rd]_g & X\ar[d]^f \\
 & Z
}\]
As $X$ is a normal surface, a general fibre of $f$ is a smooth projective curve.
Thus, $\phi$ is an isomorphism on general fibres of $g$ and $f$ by generic smoothness.  
Let $F$ be a general fibre of $g$.  
Since $\mathbf{M}_Y$ is  
a nef $\QQ$-divisor, we have $F\cdot \mathbf{M}_Y \ge 0$.
Moreover, as $H_Z$ is an ample $\QQ$-divisor on $Z$, $\mathbf{A}_Y = g^* H_Z$
is $\QQ$-linearly equivalent to a finite $\QQ^{>0}$-linear
combination of general fibres of $g$.
If $F\cdot \mathbf{M}_Y > 0$, then $\mathbf{A}_Y\cdot \mathbf{M}_Y > 0$,
and hence we have
\[ (\mathbf{A}_Y + \mathbf{M}_Y)^2 = (\mathbf{A}_Y)^2 + 2\mathbf{A}_Y\cdot \mathbf{M}_Y
+ (\mathbf{M}_Y)^2 > 0. \]
Thus, $\mathbf{A}_Y + \mathbf{M}_Y$ is nef and big, which implies that 
the injectivity theorem holds for $\{(X, B + \mathbf{M}), L\}$ by 
Theorem~\ref{b-kodaira-non-negative} (2). \\

We assume that $\deg (\mathbf{M}_Y|_F) = 0$ in the rest of the proof.\\


\emph{Step 3}.
If $\hg^0(F, \mc{O}_F(n\mathbf{M}_Y|_F)) \not= 0$ for some $n\in \NN$, 
then $\mathbf{M}_Y|_F \sim_{\QQ} 0$ on $F$.
In this case, we have the equality
$$
\kappa (\mathbf{M}_Y|_F) = \nu (\mathbf{M}_Y|_F) = \deg (\mathbf{M}_Y|_F) = 0,
$$
which implies that $\mathbf{M}$ is b-nef-and-abundant over $Z$.  
Let $\varphi\colon \wt{Y}\to Y$ be a log resolution on which $\mathbf{M}$
descends to a nef $\QQ$-divisor $\mathbf{M}_{\wt{Y}}$.
Then $\mathbf{M}_{\wt{Y}}$ is nef and $(g\circ \varphi)$-abundant.
Recall that $\mathbf{A}_{\wt{Y}} = (g\circ \varphi)^* H_Z$
for an ample $\QQ$-divisor $H_Z$ on $Z$.
So, by \cite[Corollary V.2.28]{nakayama},
$\mathbf{A}_{\wt{Y}} + \mathbf{M}_{\wt{Y}}$ is nef and abundant.
By Lemma~\ref{abundant-to-nef-big}, there exist a resolution $\pi\colon W\to \wt{Y}$, 
a contraction $\tau\colon W\to C$ of non-singular projective varieties,
and a nef and big $\QQ$-divisor $B_C$ on $C$ such that 
\[ \pi^*(\mathbf{A}_{\wt{Y}} + \mathbf{M}_{\wt{Y}})\sim_{\QQ} \tau^* B_C.  \]
Note that $\nu (\mathbf{A}_{\wt{Y}} + \mathbf{M}_{\wt{Y}}) \ge 1$ by construction.
If $\nu (\mathbf{A}_{\wt{Y}} + \mathbf{M}_{\wt{Y}}) = 2$, then
$\mathbf{A}_{\wt{Y}} + \mathbf{M}_{\wt{Y}}$ is a nef and big divisor
by \cite[Lemma V.2.1 (2)]{nakayama}.
Therefore, the injectivity theorem holds for $\{(X, B + \mathbf{M}), L\}$ by 
Theorem~\ref{b-kodaira-non-negative} (2).

So, we can assume that $\nu (\mathbf{A}_{\wt{Y}} + \mathbf{M}_{\wt{Y}}) = 1$.
As numerical dimension of a nef divisor is invariant up to pulling back via 
proper surjective morphisms (see \cite[Proposition V.2.7]{nakayama}),
the numerical dimension $\nu (B_C)$ is also equal to one.
Since $B_C$ is nef and big, we conclude that $C$ is a smooth projective curve
and that $B_C$ is an ample $\QQ$-divisor on $C$.
In particular, $\pi^*(\mathbf{A}_{\wt{Y}} + \mathbf{M}_{\wt{Y}})$ is semi-ample.
We can assume that $\tau\colon W\to C$ is the contraction induced by 
$\pi^*(\mathbf{A}_{\wt{Y}} + \mathbf{M}_{\wt{Y}})$.  
Let $F'$ be a general fibre of $g\circ \varphi\circ \pi\colon W\to Z$, 
a smooth projective curve that is isomorphic 
to a general fibre $\wt{F}$ of $g\circ \varphi$
(also to a general fibre $F$ of $g$).  
Since $F\cdot \mathbf{A}_Y = 0$ and since $\deg (\mathbf{M}_Y|_F) = 0$, we have
\[ F'\cdot \pi^*(\mathbf{A}_{\wt{Y}} + \mathbf{M}_{\wt{Y}}) 
= \wt{F}\cdot (\mathbf{A}_{\wt{Y}} + \mathbf{M}_{\wt{Y}}) = 0. \]
Thus, $F'$ is contracted by $\tau$.  
By the rigidity lemma (see \cite[Lemma 1.6]{km98} 
and \cite[Lemma 1.15]{Debarre_higher_dim_ag}), there exists a non-constant morphism 
$h\colon Z\to C$ such that $\tau$ factors through $g\circ\varphi\circ \pi$ via $h$.
As all the morphisms are contractions, $h$ is an isomorphism.
So, for a sufficiently small rational number $0<\varepsilon <1$, we have
\[ \pi^*(\mathbf{A}_{\wt{Y}} + \mathbf{M}_{\wt{Y}}) - \varepsilon \pi^*(\varphi^*\phi^* D) \]
is $\QQ$-effective, hence 
$(\mathbf{A}_{\wt{Y}} + \mathbf{M}_{\wt{Y}}) - \varepsilon (\phi \circ \varphi)^*D$ 
is also $\QQ$-effective.
Combining with Step 2 in the proof of Theorem~\ref{b-kodaira-non-negative}
(up to replacing $\wt{Y}$ by higher log resolutions), we see that
the injectivity theorem holds for $\{(X, B + \mathbf{M}), L\}$. \\

We assume that $\hg^0(F, \mc{O}_F(n\mathbf{M}_Y|_F)) = 0$ for any $n\in \NN$,
that is, $\kappa (\mathbf{M}_Y|_F) = -\infty$ in the rest of the proof.
In particular, we have assumed that $\mathbf{M}_Y|_F \not\sim_{\QQ} 0$. \\


\emph{Step 4}.
Now, we can compute $\hg^1(F, \mc{O}_F (K_Y + B_Y + \mathbf{M}_Y + \mathbf{A}_Y))$ as follows.
As $F$ is a general fibre, we have $K_Y|_F = K_F$; moreover, $\mathbf{A}_Y|_F = 0$ since
$g$ is the contraction induced by $\mathbf{A}_Y$.  Therefore, Serre duality implies that
\begin{align*}
	\hg^1(F, \mc{O}_F (K_Y + B_Y + \mathbf{M}_Y + \mathbf{A}_Y)) &= \hg^1(F, \mc{O}_F( K_F + B_Y|_F + \mathbf{M}_Y|_F )) \\
	  &\simeq \hg^0(F, \mc{O}_F( -B_Y|_F - \mathbf{M}_Y|_F ))^{\vee}.
\end{align*}
Recall that $B_Y|_F$ is an effective divisor.
Thus, we conclude that 
\[ h^0(F, \mc{O}_F( -B_Y|_F - \mathbf{M}_Y|_F ) )\le h^0(F, \mc{O}_F( - \mathbf{M}_Y|_F )).\]
However, $h^0(F, \mc{O}_F( - \mathbf{M}_Y|_F)) = 0$ as
$\deg \mathbf{M}_Y|_F = 0$ and $\mathbf{M}_Y|_F \not\sim_{\QQ} 0$. 
Then we have
\begin{equation}
	\hg^1(F, \mc{O}_F (K_Y + B_Y + \mathbf{M}_Y + \mathbf{A}_Y)) = 0. \label{ch-vanising}
\end{equation}\\


\emph{Step 5}.
Denote by $L_Y$ the Cartier divisor
\[ \phi^* L = K_Y + B_Y + \mathbf{M}_Y + \mathbf{A}_Y. \]
By \cite[Lemma II.2.11]{nakayama}, we have the canonical isomorphism
\[ \mc{O}_X(L) \simeq \phi_* \mc{O}_Y(L_Y). \]
Consider the Leray spectral sequence 
\[ E_2^{pq} := \hg^p(X, R^q\phi_* \mc{O}_Y(L_Y)) \Rightarrow \hg^{p+q}(Y, \mc{O}_Y(L_Y)). \]
Taking $q=0$, we get the edge morphism of the spectral sequence
\[ e^p\colon \hg^p(X, \mc{O}_X(L)) \to \hg^p(Y, \mc{O}_Y(L_Y)). \]
The five-term exact sequence shows that $e^0$ is an isomorphism
and that $e^1$ is injective.

By \cite[Exercise 5.3.13 (b)]{Liuqing}, there is an isomorphism
\[ R^1g_* \mc{O}_Y(L_Y)\otimes_{\mc{O}_Z} k(z)\simeq 
\hg^1(Y_z, \mc{O}_{Y_z}(L_Y) ) \]
for every point $z\in Z$ (not necessarily a closed point).  
Since $R^1g_* \mc{O}_Y(L_Y)$ is a coherent sheaf on $Z$, \eqref{ch-vanising}
implies that $R^1g_* \mc{O}_Y(L_Y)$ is a skyscraper sheaf on $Z$.

Denote by $D_Y\in \abs{m \mathbf{A}_Y}$ the pullback divisor $\phi^* D$. 
By construction, every component of $D_Y$ is vertical$/Z$.
Then by varying $D$ generally for $m\gg 0$ (see Lemma~\ref{suff-large-m}), we can assume that $D_Y$
consists of finitely many (smooth) general fibres of $g$. 
In particular, we can assume that $g_*(D_Y)$ avoids the support of $R^1g_* \mc{O}_Y(L_Y)$,
which is independent of the choice of the divisor $D$.\\


\emph{Step 6}.
Consider the exact sequence of invertible sheaves
\begin{equation}
	0\to \mc{O}_Y(L_Y) \to \mc{O}_Y (L_Y + D_Y) \to \mc{O}_{D_Y} (L_Y + D_Y) \to 0.\label{basic-ses}
\end{equation}
Taking cohomologies gives the long exact sequence
\begin{multline}
0\to \hg^0(Y, \mc{O}_Y(L_Y)) \to \hg^0(Y, \mc{O}_Y (L_Y + D_Y)) 
     \to \hg^0(Y, \mc{O}_{D_Y} (L_Y + D_Y)) \\
	 \to \hg^1(Y, \mc{O}_Y(L_Y)) 
	 \to \hg^1(Y, \mc{O}_Y (L_Y + D_Y)) \to \hg^1(Y, \mc{O}_{D_Y} (L_Y + D_Y)) \\
	 \to \hg^2(Y, \mc{O}_Y(L_Y)) \to \hg^2(Y, \mc{O}_Y (L_Y + D_Y)) \to 0.\label{long-1}
\end{multline}
By the assumption in Step 5, $D_Y$ consists of a finite union of general fibres of $g$, hence 
$\hg^1(Y, \mc{O}_{D_Y} (L_Y + D_Y))$ is a direct sum of $\hg^1(F, \mc{O}_{F} (L_Y + D_Y))$.
Moreover, as $D_Y|_F = 0$, the sheaf $\mc{O}_F(L_Y + D_Y)$ is equal to $\mc{O}_F(L_Y)$
whose first cohomology vanishes by \eqref{ch-vanising}.
Therefore, we have that
\[ \hg^2(Y, \mc{O}_Y(L_Y)) \simeq \hg^2(Y, \mc{O}_Y (L_Y + D_Y)) \]
\begin{align*}
    \text{and }\,	0&\to \hg^0(Y, \mc{O}_Y(L_Y)) \to \hg^0(Y, \mc{O}_Y (L_Y + D_Y)) \to \hg^0(Y, \mc{O}_{D_Y} (L_Y + D_Y))\\
	 &\to \hg^1(Y, \mc{O}_Y(L_Y)) \to \hg^1(Y, \mc{O}_Y (L_Y + D_Y)) \to 0.
\end{align*}
On the other hand, as $\phi\colon Y\to X$ is an isomorphism near 
a general fibre $F$ of $g\colon Y\to Z$,
the computation also applies to the exact sequence of invertible sheaves on $X$,
\[ 0\to \mc{O}_X(L) \to \mc{O}_X (L + D) \to \mc{O}_{D} (L + D) \to 0, \]
where $D$ consists of a finite union of general fibres of $f\colon X\to Z$.
In particular, we have
\[ \hg^2(X, \mc{O}_X(L)) \simeq \hg^2(X, \mc{O}_X(L+D)). \]
Notice that we have the following commutative diagram,
\[\xymatrix{
\hg^1(X, \mc{O}_X(L))\ar[d]\ar[r]^-{e_1} & \hg^1(Y, \mc{O}_Y(L_Y))\ar[d] \\
\hg^1(X, \mc{O}_X(L+D))\ar[r] & \hg^1(Y, \mc{O}_Y(L_Y + D_Y))
}\]
where the vertical morphisms are the multiplication homomorphisms 
on cohomologies.
Therefore, since $e^1$ is injective (see Step 5), it remains to show that
the multiplication homomorphism of first cohomologies on $Y$,
$$\hg^1(Y, \mc{O}_Y(L_Y)) \to \hg^1(Y, \mc{O}_Y(L_Y + D_Y)),$$
is injective.\\


\emph{Step 7}.
Taking higher direct images of the exact sequence \eqref{basic-ses} gives that 
\[ 0\to g_*\mc{O}_Y(L_Y) \to g_*\mc{O}_Y (L_Y + D_Y) \to g_*\mc{O}_{D_Y} (L_Y + D_Y) \to R^1g_*\mc{O}_Y(L_Y)\to \cdots.  \]
However, by Step 5, the divisor $g_*(D_Y)$ avoids the support of $R^1g_* \mc{O}_Y(L_Y)$, 
hence we have the exact sequence
\[ 0\to g_*\mc{O}_Y(L_Y) \to g_*\mc{O}_Y (L_Y + D_Y) \to g_*\mc{O}_{D_Y} (L_Y + D_Y) \to 0. \]
Then there is a long exact sequence of cohomologies
\begin{multline}
	0\to \hg^0(Y, \mc{O}_Y(L_Y)) \to \hg^0(Y, \mc{O}_Y (L_Y + D_Y)) \to \hg^0(Y, \mc{O}_{D_Y} (L_Y + D_Y))  \\
	 \to \hg^1(Z, g_*\mc{O}_Y(L_Y)) \to \hg^1(Z, g_*\mc{O}_Y (L_Y + D_Y)) \to \hg^1(Z, g_*\mc{O}_{D_Y} (L_Y + D_Y))\to \cdots. \label{long-2}
\end{multline}
Recall that $D_Y$ consists of finitely many (smooth) general fibres of $g$,
so $D_Y$ is vertical$/Z$.  Thus, $g_*\mc{O}_{D_Y} (L_Y + D_Y)$ supports
on finitely many closed points of the smooth curve $Z$.
By \cite[Theorem III.2.7]{Hart}, we have that
\[ \hg^1(Z, g_*\mc{O}_{D_Y} (L_Y + D_Y)) = 0. \]
Comparing the exact sequences \eqref{long-1} and \eqref{long-2},
we see that it suffices to show that 
\[ \hg^1(Z, g_* \mc{O}_Y(L_Y)) = 0. \]
Notice that we can write
\[ g_*\mc{O}_Y(L_Y) = g_*\mc{O}_Y(K_Y + B_Y + \mathbf{M}_Y + \mathbf{A}_Y)
= \omega_Z\otimes_{\mc{O}_Z} g_*\mc{O}_Y(K_{Y/Z} + B_Y + \mathbf{M}_Y + \mathbf{A}_Y) \]
by projection formula (cf. \cite[Corollary 24]{Kleiman80}).  So, it suffices to show that
\[ \hg^1(Z, \omega_Z\otimes_{\mc{O}_Z} g_*(\omega_{Y/Z}\otimes_{\mc{O}_Y} \mc{O}_Y(B_Y + \mathbf{M}_Y + \mathbf{A}_Y))) = 0. \]
Set 
\[ \mc{G}_Y := \mc{O}_Y(B_Y + \mathbf{M}_Y + \mathbf{A}_Y). \]
The vanishing is trivial if 
$g_*(\omega_{Y/Z}\otimes_{\mc{O}_Y} \mc{G}_Y)$ is zero,
so we assume that $g_*(\omega_{Y/Z}\otimes_{\mc{O}_Y} \mc{G}_Y)\not=0$.  
Moreover, as $Z$ is a smooth curve, 
the coherent sheaf $g_*(\omega_{Y/Z}\otimes_{\mc{O}_Y} \mc{G}_Y)$
is locally free, in other words, $g_*(\omega_{Y/Z}\otimes_{\mc{O}_Y} \mc{G}_Y)$ is a vector bundle.

By the vanishing theorem for Nakano-positive vector bundles (see \cite[page 97]{robpositivity2}),
it suffices to show that the vector bundle
$g_*(\omega_{Y/Z}\otimes_{\mc{O}_Y} \mc{G}_Y)$
is Nakano-positive.  However, by \cite[Theorem 2.6]{Umemura73}, since $Z$ 
is a smooth projective curve (i.e., a Riemann surface), Nakano-positivity of
vector bundles on $Z$ is equivalent to ampleness.  So, it suffices to show that 
$g_*(\omega_{Y/Z}\otimes_{\mc{O}_Y} \mc{G}_Y)$
is an ample vector bundle on $Z$, that is,
every quotient vector bundle of $g_*(\omega_{Y/Z}\otimes_{\mc{O}_Y} \mc{G}_Y)$ 
has positive degree (see \cite[\S 2]{Hart-ample-vbdl}).\\


\emph{Step 8}.
Recall that $0\le B_Y<1$ is an effective $\QQ$-divisor (see Step 1).
Let $(\mc{L}_{B_Y}, h_{B_Y})$ be a hermitian $\QQ$-line bundle
representing the canonical hermitian class of $B_Y$ (see \S \ref{cano-h-class-divisors}).
Then $\Theta_{h_{B_Y}}(\mc{L}_{B_Y}) \ge 0$ by Remark~\ref{eff-Q-div-h-metric}.
Furthermore, by Lemma~\ref{res-pullback-cano-metric}, 
$(\mc{L}_{B_Y}|_F, h_{B_Y}|_F)$ is a hermitian $\QQ$-line bundle
representing the canonical hermitian class of the $\QQ$-divisor $B_Y|_F$.
Since $B_Y<1$,
the multiplier ideal sheaf $\mc{I}(h_{B_Y}|_F)$ is trivial
by \cite[Theorem 9.3.42, Proposition 9.5.13]{robpositivity2}.
By the Ohsawa-Takegoshi extension theorem (see \cite[Proposition 1.3, Remark 3.2.3]{PT18}), we have the inclusion 
\[ \mc{O}_F\simeq \mc{I}(h_{B_Y} |_F)  \subseteq 
\mc{I}(h_{B_Y})\cdot \mc{O}_F. \]
Thus, the restriction of the multiplier ideal sheaf $\mc{I}(h_{B_Y})|_F$ is trivial.

Denote by $\mc{L}_{\mathbf{M}_Y}$ a $\QQ$-line bundle in the $\QQ$-line 
bundle class $\Xi ([\mathbf{M}_Y])$ (see Lemma~\ref{Q-div-Q-line-bdl}).
Take $r\in \NN$ such that $r\mathbf{M}_Y$ is a nef Cartier divisor.
Then the line bundle $\mc{O}_Y(r\mathbf{M}_Y)$ admits a singular hermitian metric $H$
with semi-positive curvature (see \cite[Proposition 4.2]{De90}).
Hence, by Lemma~\ref{fraction-metric}, the $\QQ$-line bundle 
$\mc{L}_{\mathbf{M}_Y}$ can be equipped with
a singular hermitian metric $h_{\mathbf{M}_Y}$ satisfying that
\[ \mc{I}(H)\subseteq \mc{I}(h_{\mathbf{M}_Y}) \text{ and } 
\Theta_{h_{\mathbf{M}_Y}}(\mc{L}_{\mathbf{M}_Y})\ge 0. \]
By Definition-Lemma~\ref{pullback-curvature}, 
$(\mc{L}_{\mathbf{M}_Y}|_F, h_{\mathbf{M}_Y}|_F)$ is a hermitian $\QQ$-line bundle on $F$
with semi-positive curvature.
However, as $\deg (r \mathbf{M}_Y)|_F = r\deg \mathbf{M}_Y|_F = 0$, 
the singular hermitian metric $h_{\mathbf{M}_Y} |_F$ on $F$ is 
actually smooth (see \cite[Lemma 13.2]{Schnell18}), hence $\mc{I}(h_{\mathbf{M}_Y} |_F) \simeq \mc{O}_F$.
Again, by the Ohsawa-Takegoshi extension theorem, there is an inclusion 
\[ \mc{O}_F\simeq \mc{I}(h_{\mathbf{M}_Y} |_F)  \subseteq 
\mc{I}(h_{\mathbf{M}_Y})\cdot \mc{O}_F, \]
so we conclude that $\mc{I}(h_{\mathbf{M}_Y})|_F$ is trivial on $F$.

Let $A_Z$ be the ample $\QQ$-divisor on $Z$ (not necessarily effective)
such that $\mathbf{A}_Y = g^*A_Z$.
Let $\alpha\in \NN$ be a positive integer such that $\alpha A_Z$ is an ample Cartier divisor.
Since $\mc{O}_Z(\alpha A_Z)$ is an ample line bundle, there exist
(by \cite[Proposition 4.2]{De90})
\begin{itemize}
	\item a K\"ahler form $\omega$ on $Z$,
	\item a positive real number $\varepsilon > 0$, and
	\item a \emph{smooth} hermitian metric $H_Z$ on $\mc{O}_Z(\alpha A_Z)$ such that 
	      \[ \Theta_{H_Z}(\mc{O}_Z(\alpha A_Z)) \ge \varepsilon \omega. \]
\end{itemize} 
Let $\mc{L}_{A_Z}$ be a $\QQ$-line bundle in the class
$\Xi ([A_Z])$ such that $\mc{L}_{A_Z}^{\otimes \alpha} \simeq \mc{O}_Z(\alpha A_Z)$ 
(see Lemma~\ref{Q-div-Q-line-bdl}).
Then by Lemma~\ref{fraction-metric}
there exists a smooth hermitian metric $h_{A_Z}$ on $\mc{L}_{A_Z}$
such that 
$$\Theta_{h_{A_Z}}(\mc{L}_{A_Z}) = \frac{1}{\alpha}\Theta_{H_Z}(\mc{O}_Z(\alpha A_Z)) 
\ge \frac{\varepsilon}{\alpha} \omega.  $$
Denote by $(\mc{L}_{\mathbf{A}_Y}, h_{\mathbf{A}_Y})$ the pullback 
hermitian $\QQ$-line bundle of $(\mc{L}_{A_Z}, h_{A_Z})$ as in 
Definition-Lemma~\ref{pullback-curvature}.
Then $\mc{L}_{\mathbf{A}_Y}$ is a $\QQ$-line bundle in the class $\Xi ([\mathbf{A}_Y])$.
Moreover, we have
\[ \Theta_{h_{\mathbf{A}_Y}}(\mc{L}_{\mathbf{A}_Y}) = g^{\star}\Theta_{h_{A_Z}}(\mc{L}_{A_Z}) =
\frac{1}{\alpha}g^{\star}\Theta_{H_Z}(\mc{O}_Z(\alpha A_Z)) 
\ge \frac{\varepsilon}{\alpha} g^{\star}\omega.  \]
Furthermore, by construction, the multiplier ideal sheaf $\mc{I}(h_{\mathbf{A}_Y})$ is trivial.

Note that $B_Y + \mathbf{M}_Y + \mathbf{A}_Y$ is a Cartier divisor on $Y$.
Thus, by Lemma~\ref{sum-metric}, 
\[ h_{\mc{G}_Y} := h_{B_Y}\cdot h_{\mathbf{M}_Y} \cdot h_{\mathbf{A}_Y} \]
is a well-defined singular hermitian metric on the integral line bundle $\mc{G}_Y$ such that 
\[ \Theta_{h_{\mc{G}_Y}}(\mc{G}_Y) = \Theta_{h_{B_Y}}(\mc{L}_{B_Y})
+ \Theta_{h_{\mathbf{M}_Y}}(\mc{L}_{\mathbf{M}_Y}) 
+ \Theta_{h_{\mathbf{A}_Y}}(\mc{L}_{\mathbf{A}_Y}) \ge 
\frac{1}{\alpha}g^{\star}\Theta_{H_Z}(\mc{O}_Z(\alpha A_Z)). \]
On the other hand, the multiplier ideal sheaves satisfy that (by Lemma~\ref{sum-metric})
\[ \mc{I}(h_{B_Y}\cdot h_{\mathbf{M}_Y} \cdot h_{\mathbf{A}_Y})\subseteq 
\mc{I}(h_{B_Y})\cdot \mc{I}(h_{\mathbf{M}_Y}) \cdot \mc{I}(h_{\mathbf{A}_Y}). \]
As both $h_{\mathbf{M}_Y}|_F$ and $h_{\mathbf{A}_Y}|_F$ are smooth 
hermitian metrics, we have (see Theorem~\ref{sub-add})
\[ \mc{O}_F = \mc{I}(h_{B_Y}|_F) = 
\mc{I}(h_{B_Y}|_F\cdot h_{\mathbf{M}_Y}|_F \cdot h_{\mathbf{A}_Y}|_F)
\subseteq \mc{I}(h_{B_Y}\cdot h_{\mathbf{M}_Y} \cdot h_{\mathbf{A}_Y})\cdot \mc{O}_F, \]
that is, $\mc{I}(h_{B_Y}\cdot h_{\mathbf{M}_Y} \cdot h_{\mathbf{A}_Y}) = \mc{I}(h_{\mc{G}_Y})$ is trivial
on a general fibre $F$ of $g$.  
Thus, all the conditions in Proposition~\ref{pushf-positivity}
are satisfied.  

Let $\mc{E}$ be an arbitrary quotient vector bundle of $g_*(\omega_{Y/Z}\otimes_{\mc{O}_Y}\mc{G}_Y)$.
Then Proposition~\ref{pushf-positivity} implies that there exists a 
singular hermitian metric $h_{\mc{E}}$ on $\mc{E}$ such that
\[ \Theta_{\det h_{\mc{E}}}(\det \mc{E}) \ge 
\frac{1}{\alpha}\Theta_{H_Z}(\mc{O}_Z(\alpha A_Z))
\ge \frac{\varepsilon}{\alpha} \omega. \]
As the image of $\Theta_{H_Z}(\mc{O}_Z(\alpha A_Z))$ in $\hg^2(Z, \CC)$ is 
equal to $\deg \mc{O}_Z(\alpha A_Z) = \deg (\alpha A_Z)$ 
(see \cite[page 192]{robpositivity2} and \cite[page 144]{Griffiths-Harris}),
we have
\[ \alpha \deg \mc{E} := \alpha \deg (\det \mc{E}) \ge \deg (\alpha A_Z) > 0, \]
since $\omega$ is a K\"ahler form on $Z$.  
This shows that $\deg \mc{E} > 0$ as desired.
\end{proof}

\medskip

\printbibliography








 \vspace{1em}
 
\noindent\small{Santai Qu} 

\noindent\small{\textsc{Institute of Geometry and Physics, University of Science and Technology of China, No. 96 Jinzhai Road, Hefei, Anhui Province, 230026, China} }

\noindent\small{Email: \texttt{santaiqu@ustc.edu.cn}}

 \end{document}